\RequirePackage{fix-cm}
\documentclass[english]{article}
\usepackage[T1]{fontenc}
\usepackage[latin9]{inputenc}
\usepackage{units}
\usepackage{textcomp}
\usepackage{amsthm}
\usepackage{amsmath}
\usepackage{amssymb}
\usepackage{fixltx2e}
\usepackage{esint}

\makeatletter
\numberwithin{equation}{section}
  \theoremstyle{remark}
  \newtheorem*{acknowledgement*}{\protect\acknowledgementname}
\theoremstyle{plain}
\newtheorem{thm}{\protect\theoremname}
  \theoremstyle{definition}
  \newtheorem{defn}[thm]{\protect\definitionname}
  \theoremstyle{plain}
  \newtheorem{prop}[thm]{\protect\propositionname}
  \theoremstyle{remark}
  \newtheorem{rem}[thm]{\protect\remarkname}
  \theoremstyle{remark}
  \newtheorem{claim}[thm]{\protect\claimname}
  \theoremstyle{plain}
  \newtheorem{cor}[thm]{\protect\corollaryname}
  \theoremstyle{plain}
  \newtheorem{conjecture}[thm]{\protect\conjecturename}

\usepackage{bbm}
\theoremstyle{plain}
\newtheorem{open}[thm]{Open Problem}
\numberwithin{thm}{subsection}
\usepackage{undertilde}

\DeclareMathOperator{\im}{im}

\DeclareMathOperator{\Spec}{Spec}
\date{}

\makeatother

\usepackage{babel}
  \providecommand{\acknowledgementname}{Acknowledgement}
  \providecommand{\claimname}{Claim}
  \providecommand{\conjecturename}{Conjecture}
  \providecommand{\corollaryname}{Corollary}
  \providecommand{\definitionname}{Definition}
  \providecommand{\propositionname}{Proposition}
  \providecommand{\remarkname}{Remark}
\providecommand{\theoremname}{Theorem}

\begin{document}

\title{Ramanujan Complexes\\
and High Dimensional Expanders%
\thanks{This paper is based on notes prepared for the Takagi Lectures, Tokyo
2012.%
}}

\author{Alexander Lubotzky\\
{\small Einstein Institute of Mathematics}\\
{\small Hebrew University}\\
{\small Jerusalem 91904 ISRAEL}\\
{\small alex.lubotzky@mail.huji.ac.il}}
\maketitle
\begin{abstract}
Expander graphs in general, and Ramanujan graphs in particular, have
been of great interest in the last three decades with many applications
in computer science, combinatorics and even pure mathematics. In these
notes we describe various efforts made in recent years to generalize
these notions from graphs to higher dimensional simplicial complexes.
\end{abstract}
\tableofcontents{}

\section{Introduction}

Expander graphs are highly connected finite sparse graphs. These graphs
play a fundamental role in computer science and combinatorics (cf.
\cite{lubotzky1994discrete,Hoory2006}, and the references within)
and in recent years even found numerous applications in pure mathematics
(\cite{Lubotzky2012}). Among these graphs, Ramanujan graphs stand
out as optimal expanders (at least from the spectral point of view).
The theory of expanders and Ramanujan graphs has led to a very fruitful
interaction between mathematics and computer science (and between
mathematicians and computer scientists). In the early days, deep mathematics
(e.g. Kazhdan property (T) and Ramanujan conjecture) has been used
to construct expanders and Ramanujan graphs. But recently, the theory
of computer science pays its debt to mathematics and expanders start
to appear more and more also within pure mathematics.

The fruitfulness of this theory calls for a generalization to high
dimensional theory. Here the theory is much less developed. The goal
of these notes is to describe some of these efforts and to call the
attention of the mathematical and computer science communities to
this challenge. We strongly believe that a beautiful and useful theory
is waiting for us to be explored.

Most of the notes will be dedicated to the story of Ramanujan complexes.
These generalizations of Ramanujan graphs, which has been developed
in \cite{cartwright2003ramanujan,li2004ramanujan,Lubotzky2005a,Lubotzky2005b,sarveniazi2007explicit}
became possible by the significant development of the theory of automorphic
forms in positive characteristic and especially the work of L.\ Lafforgue
\cite{lafforgue2002chtoucas}. In §1, we will describe the classical
theory of Ramanujan graphs, in a way which will pave the way for a
smooth presentation in §2, of the much more complicated theory of
Ramanujan complexes.

The situation with high dimensional expanders is more chaotic. Here
it is not even agreed what should be the ``right'' definition. Several
generalizations of the concept of expander graph have been suggested,
which are not equivalent. It is not clear at this point which one
is more useful. Each has its own charm and part of the active research
on this subject is to understand the relationships between the various
definitions.

We describe these activities briefly in §3. It can be expected (and,
in fact, I hope!) that these notes will not be up to date by the time
they will appear in press\ldots{}
\begin{acknowledgement*}
The author is indebted to Konstantin Golubev, Gil Kalai, Tali Kaufman,
Roy Meshulam and Uli Wagner for many discussions regarding the material
of these notes. I am especially grateful to Ori Parzanchevski, whose
help and advice in preparing these notes have improved them significantly.

This work was supported by ERC, ISF and NSF. Some of this work was
carried out while the author visited Microsoft Research laboratory
in Cabridge, Ma. We are grateful for the warm hospitality and support.
We are also grateful for the Mathematical Society of Japan for the
invitation to deliver the Takagi lectures, the hospitality and the
valuable remarks we received from the audience.
\end{acknowledgement*}

\section{Ramanujan Graphs}

In this chapter we will survey Ramanujan graphs, which are optimal
expanding graphs from a spectral point of view. The material is quite
well known by now and has been described in various places (\cite{Lubotzky1988,Sarnak1990,lubotzky1994discrete,Valette1997}).
We present it here in a way which will pave the way for the high dimensional
generalization - the Ramanujan complexes - which will come in the
next chapter.

\subsection{\label{sub:Eigenvalues-and-expanders}Eigenvalues and expanders}

Let $X=\left(V,E\right)$ be a finite connected $k$-regular graph,
$k\geq3$, with a set $V$ of $n$ vertices, and adjacency matrix
$A=A_{X}$, i.e.\ $A$ is an $n\times n$ matrix indexed by the vertices
of $X$ and $A_{ij}$ is equal to the number of edges between $i$
and $j$ (which is either $0$ or $1$ if $X$ is a simple graph).
\begin{defn}
The graph $X$ is called \emph{Ramanujan }if for every eigenvalue
$\lambda$ of the symmetric matrix $A$, either $\lambda=\pm k$ (``the
trivial eigenvalues'') or $\left|\lambda\right|\leq2\sqrt{k-1}$.
\end{defn}
Recall that $k$ is always an eigenvalue of $A$ (with the constant
vector/function as an eigenfunction) while $-k$ is an eigenvalue
of $A$ iff $X$ is bi-partite, i.e.\ the vertices of $X$ can be
divided into two disjoint sets $Y$ and $Z$ and every edge $e$ in
$E$, has one endpoint in $Y$ and one in $Z$. In this case, the
eigenfunction is $1$ on $Y$ and $-1$ on $Z$.

Ramanujan graphs have been defined and constructed in \cite{Lubotzky1988}
(see also \cite{margulis1988explicit} and see \cite{Sarnak1990,lubotzky1994discrete,Valette1997}
for more comprehensive treatment). The importance of the number $2\sqrt{k-1}$
comes from Alon-Boppana Theorem which asserts that for any fixed $k$,
no better bound can be obtained on the non-trivial eigenvalues of
an infinite sequence of finite $k$-regular graphs.
\begin{thm}[Alon-Boppana (cf. \cite{Lubotzky1988,Nilli1991})]
\label{thm:Alon-Boppana}For a finite connected $k$-regular graph
$X$, denote 
\begin{align*}
\mu_{1}\left(X\right) & =\max\left\{ \lambda\,\middle|\,\lambda\mathrm{\: an\: eigenvalue\: of\:}A\mathrm{\: and}\:\lambda\neq k\right\} \\
\mu_{0}\left(X\right) & =\max\left\{ \left|\lambda\right|\,\middle|\,\lambda\mathrm{\: an\: eigenvalue\: of\:}A\mathrm{\: and}\:\lambda\neq k\right\} \\
\mu\left(X\right) & =\max\left\{ \left|\lambda\right|\,\middle|\,\lambda\mathrm{\: an\: eigenvalue\: of\:}A\mathrm{\: and}\:\lambda\neq\pm k\right\} .
\end{align*}
If $\left\{ X_{i}\right\} _{i=1}^{\infty}$ is a sequence of such
graphs with $\left|X_{i}\right|\rightarrow\infty$, then 
\[
\liminf_{i\rightarrow\infty}\mu\left(X_{i}\right)\geq2\sqrt{k-1}.
\]

\end{thm}
The hidden reason for the number $2\sqrt{k-1}$ is: All the finite
connected $k$-regular graphs are covered by the $k$-regular tree,
$T=T_{k}$. Let $A_{T}$ be the adjacency operator of $T$, i.e.,
for every function $f$ on the vertices of $T$ and for every vertex
$x$ of it, 
\[
A_{T}\left(f\right)\left(x\right)=\sum_{y\sim x}f\left(y\right)
\]
namely, $A_{T}$ sums $f$ over the neighbors of $x$. Then $A_{T}$
defines a self adjoint operator $L^{2}\left(T\right)\rightarrow L^{2}\left(T\right)$.
\begin{prop}[\cite{kesten1959symmetric}]
\label{prop:tree-spec}The spectrum of $A_{T}$ is $\left[-2\sqrt{k-1},2\sqrt{k-1}\right]$.
\end{prop}
Of course, $k$ is not an eigenvalue of $A_{T}$ as the constant function
is not in $L^{2}$. It is even not in the spectrum (unless $k=2$,
in which case $T_{k}$ is a Cayley graph of the amenable group $\mathbb{Z}$,
but this is a different story). But, $k$ is necessarily an eigenvalue
for all the adjacency operators induced on the finite quotients $\Gamma\backslash T$,
where $\Gamma$ is a discrete cocompact subgroup of $\mathrm{Aut}\left(T\right)$.
Similarly, $-k$ is an eigenvalue of the finite quotient $\Gamma\backslash T$
if it is bi-partite (which happens if $\Gamma=\pi_{1}\left(\Gamma\backslash T\right)$
preserves the two-coloring of the vertices of $T$). Now, Ramanujan
graphs are the ``ideal objects'' having their non-trivial spectrum
as good as the ``ideal object'' $T$. 

There is another way to characterize Ramanujan graphs. These are the
graphs which satisfy the ``Riemann hypothesis'', i.e.\ all the
poles of the Ihara zeta function associated with the graph lie on
the line $\Re\left(s\right)=\frac{1}{2}$. See \cite[§4.5]{lubotzky1994discrete}
and especially the works of Ihara \cite{ihara1966discrete}, Sunada
\cite{sunada1988fundamental} and Hashimoto \cite{hashimoto1989zeta}.

The work of Ihara showed the close connection between number theoretic
questions and the combinatorics of some associated graphs. While it
was Satake \cite{satake1966spherical} who showed how the classical
Ramanujan conjecture can be expressed in a representation theoretic
way. These works have paved the way to the explicit constructions
of Ramanujan graphs to be presented in §\ref{sub:Bruhat-Tits-trees}
and §\ref{sub:Explicit-constructions}.

Ramanujan graphs have found numerous applications in combinatorics,
computer science and pure mathematics. We will not describe these
but rather refer the interested readers to the thousands references
appearing in google scholar when one looks for Ramanujan graphs.

We should mention however their main application and original motivation:
expanders.
\begin{defn}
\label{def:Cheeger}For $X$ a $k$-regular graph on $n$ vertices,
denote:
\[
h\left(X\right)=\min_{0<\left|A\right|<\left|V\right|}\frac{n\cdot\left|E\left(A,V\backslash A\right)\right|}{\left|A\right|\left|V\backslash A\right|}
\]
where $E\left(A,V\backslash A\right)$ is the set of edges from $A$
to its complement. We call $h\left(X\right)$ the Cheeger constant
of $X$.\end{defn}
\begin{rem}
\label{rem:old-cheeger}In most references, the Cheeger constant is
defined as 
\[
\overline{h}\left(X\right)=\min_{0<\left|A\right|\leq\nicefrac{\left|V\right|}{2}}\frac{\left|E\left(A,V\backslash A\right)\right|}{\left|A\right|}.
\]
Clearly $\overline{h}\left(X\right)\leq h\left(X\right)\leq2\overline{h}\left(X\right)$.
For our later purpose, it will be more convenient to work with $h\left(X\right)$.\end{rem}
\begin{defn}
\label{def:expander}The graph $X$ is called $\varepsilon$-expander
(for $0<\varepsilon\in\mathbb{R}$) if $h\left(X\right)\geq\varepsilon$.
\end{defn}
Expander graphs are of great importance in computer science. Ramanujan
graphs give outstanding expanders due to the following result:
\begin{thm}[\cite{Tanner1984,Dodziuk1984,Alon1985,Alon1986}]
\label{thm:discrete-cheeger}For $X$ as above,
\[
\frac{h^{2}\left(X\right)}{8k}\leq k-\mu_{1}\left(X\right)\leq h\left(X\right).
\]
In particular, Ramanujan $k$-regular graphs are $\varepsilon$-expanders
with $\varepsilon=k-2\sqrt{k-1}$ (or if one prefers the more standard
notation $\overline{h}\left(X\right)\geq\frac{k}{2}-\sqrt{k-1}$).
\end{thm}
A very useful result in many applications is the following Expander
Mixing Lemma:
\begin{prop}
\label{prop:mixing-lemma}For $X=\left(V,E\right)$ as above and for
every two subsets $A$ and $B$ of $V$, 
\[
\left|E\left(A,B\right)-\frac{k\left|A\right|\left|B\right|}{\left|V\right|}\right|\leq\mu_{0}\left(X\right)\sqrt{\left|A\right|\left|B\right|}.
\]

\end{prop}
Note that $\frac{k\left|A\right|\left|B\right|}{\left|V\right|}$
is the expected number of edges between $A$ and $B$ if $X$ would
be a ``random $k$-regular graph''. So, if $\mu_{0}\left(X\right)$
is small, e.g.\ if $X$ is Ramanujan, it mimics various properties
of random graphs. This is one of the characteristics which make them
so useful.

There is no easy method to construct Ramanujan graphs. Let us better
be more precise here: There are many ways to get for a fixed $k$
finitely many $k$-regular Ramanujan graphs (see \cite[Chapter 8]{lubotzky1994discrete}),
but there is essentially only one known way to get, for a fixed $k$,
infinitely many $k$-regular Ramanujan graphs. The current state of
the art is, that for every $k\in\mathbb{N}$ of the form $k=p^{\alpha}+1$
where $p,\alpha\in\mathbb{N}$ and $p$ prime, there are infinitely
many $k$-regular Ramanujan graphs but for all other $k$'s this is
still open:

\begin{open}

Given $k$ which is not of the form $p^{\alpha}+1$, are there infinitely
many $k$-regular Ramanujan graphs?

\end{open}

We stress that this problem is open for every single $k$ like that
(e.g.\ $k=7$) and it is not known if such graphs exist, let alone
an explicit construction.

In the next subsection we will describe the Bruhat-Tits tree and present
the basic theory that will enable us in the following subsection to
get explicit constructions of Ramanujan graphs.

\subsection{\label{sub:Bruhat-Tits-trees}Bruhat-Tits trees and representation
theory of $\mathrm{PGL}_{2}$}

Let $F$ be a local field (e.g.\ $F=\mathbb{Q}_{p}$ the field of
$p$-adic numbers, or a finite extension of it, or $F=\mathbb{F}_{q}\left(\left(t\right)\right)$
the field of Laurent power series over the finite field $\mathbb{F}_{q}$)
with ring of integers $\mathcal{O}$ (e.g.\ $\mathcal{O}=\mathbb{Z}_{p}$
or $\mathcal{O}=\mathbb{F}_{q}\left[\left[t\right]\right]$), maximal
ideal $\mathfrak{m}=\pi\mathcal{O}$ where $\pi$ is a fixed uniformizer,
i.e., an element of $\mathcal{O}$ with valuation $\nu\left(\pi\right)=1$
(e.g.\ $\pi=p$ or $\pi=t$, respectively), so $k=\nicefrac{\mathcal{O}}{\mathfrak{m}}$
is a finite field of order $q$. Let $G=\mathrm{PGL}_{2}\left(F\right)$
and $K=\mathrm{PGL}_{2}\left(\mathcal{O}\right)$, a maximal compact
subgroup of $G$. The quotient space $\nicefrac{G}{K}$ is a discrete
set which can be identified as the set of vertices of the regular
tree of degree $q+1$ in the following way:

Let $V=F^{2}$ be the two dimensional vector space over $F$. An $\mathcal{O}$-submodule
$L$ of $V$ is called an $\mathcal{O}$-lattice if it is finitely
generated as an $\mathcal{O}$-module and spans $V$ over $F$. Every
such $L$ is of the form $L=\mathcal{O}\alpha+\mathcal{O}\beta$ where
$\left\{ \alpha,\beta\right\} $ is some basis of $V$ over $F$.
The \emph{standard lattice} is the one with $\left\{ \alpha,\beta\right\} =\left\{ e_{1},e_{2}\right\} $,
where $\left\{ e_{1},e_{2}\right\} $ is the standard basis of $V$.

Two $\mathcal{O}$-lattices $L_{1}$ and $L_{2}$ are said to be equivalent
if there exists $0\neq\lambda\in F$ such that $L_{2}=\lambda L_{1}$.
The group $\mathrm{GL}_{2}\left(F\right)$ acts transitively on the
set of $\mathcal{O}$-lattices and its center $Z$, the group of scalar
matrices, preserves the equivalent classes. Hence $G=\mathrm{PGL}_{2}\left(F\right)$
acts on these classes, with $K=\mathrm{PGL}_{2}\left(\mathcal{O}\right)$
fixing the equivalent class of the standard lattice $x_{0}=\left[L_{0}\right]$,
$L_{0}=\mathcal{O}e_{1}+\mathcal{O}e_{2}$. So, $\nicefrac{G}{K}$
can be identified with the set of equivalent classes of lattices.
Two classes $\left[L_{1}\right]$ and $\left[L_{2}\right]$ are said
to be \emph{adjacent} if there exists representatives $L_{1}'\in\left[L_{1}\right]$
and $L_{2}'\in\left[L_{2}\right]$ such that $L_{1}'\subseteq L_{2}'$
and $\nicefrac{L_{2}^{'}}{L_{1}'}\simeq k\left(=\nicefrac{\mathcal{O}}{\mathfrak{m}}\right)$.
This symmetric relation (since $\pi L_{2}'\subseteq L_{1}'$ and $\nicefrac{L_{1}'}{\pi L_{2}'}\simeq k$)
defines a structure of a graph.
\begin{thm}[{cf.\ \cite[p. 70]{serre1980trees}}]
\label{thm:tree-is-regular}The above graph is a $\left(q+1\right)$-regular
tree.
\end{thm}
The $q+1$ neighbors of $\left[L_{0}\right]$ correspond to the $q+1$
subspaces of co-dimension $1$ of the two dimensional space $\nicefrac{L_{0}}{\pi L_{0}}\cong k^{2}$.
We can therefore identify them with $\mathbb{P}^{1}\left(k\right)$,
the projective line over $k$.

Let us now shift our attention for a moment to the unitary representation
theory of $G$. Let $C=C_{c}\left(K\backslash G/K\right)$ denote
the set of bi-$K$-invariant functions on $G$ with compact support.
This is an algebra with respect to convolution:
\[
f_{1}*f_{2}\left(x\right)=\int_{G}f_{1}\left(xg\right)f_{2}\left(g^{-1}\right)dg.
\]
The algebra $C$ is commutative (see \cite[Chapter 5]{lubotzky1994discrete}
and the references therein). If $\mathcal{H}$ is a Hilbert space
and $\rho:G\rightarrow U\left(\mathcal{H}\right)$ a unitary representation
of $G$, then $\rho$ induces a representation $\overline{\rho}$
of the algebra $C$ by:
\[
\overline{\rho}\left(f\right)=\int_{G}f\left(g\right)\rho\left(g\right)dg.
\]
Let $\mathcal{H}^{K}$ be the space of $K$-invariant vectors in $\mathcal{H}$.
Then $\overline{\rho}\left(f\right)\left(\mathcal{H}^{K}\right)\subseteq\mathcal{H}^{K}$
and so $\left(\mathcal{H}^{K},\overline{\rho}\right)$ is a representation
of $C$. A basic claim is that if $\rho$ is irreducible and $\mathcal{H}^{K}\neq\left\{ 0\right\} $
then $\overline{\rho}$ is irreducible, in fact, as $C$ is commutative
Schur's Lemma implies that $\dim\mathcal{H}^{K}=1$. So $\dim\mathcal{H}^{K}=0$
or $1$, in the second case we say that $\rho$ is \emph{$K$-spherical}
(or unramified\emph{ }or of class one). We will be interested only
in these representations. Such a representation $\rho$ is uniquely
determined by $\overline{\rho}$. Let us understand now what is the
algebra $C$.

Let $\overline{\delta}$ be the characteristic function of the subset
$K\left(\begin{smallmatrix}\pi & 0\\
0 & 1
\end{smallmatrix}\right)K$ of $G$. By its definition $\overline{\delta}\in C$. In fact, it
turns out that $C$ is generated as an algebra by $\overline{\delta}$
and hence every $K$-spherical irreducible subrepresentation $\left(\mathcal{H},\rho\right)$
of $G$ is determined by the action of $\overline{\delta}$ on the
one dimensional space $\mathcal{H}^{K}$, i.e.\ by the eigenvalue
of this action.

Let us note now that $C$ also acts on $L^{2}\left(\nicefrac{G}{K}\right)$
in the following way: If $f_{1}\in C$ and $f_{2}\in L^{2}\left(\nicefrac{G}{K}\right)$
we think of both as functions on $G$ and we can then look at $f_{2}*f_{1}\in L^{2}\left(\nicefrac{G}{K}\right)$
(check!)

Spelling out the meaning of that for $f_{1}=\overline{\delta}$, one
can see (the reader is strongly encouraged to work out this exercise!):
\begin{claim}
Let $f$ be a function defined on the vertices of the tree $\nicefrac{G}{K}$
and let $\delta$ be the operator $\delta:L^{2}\left(\nicefrac{G}{K}\right)\rightarrow L^{2}\left(\nicefrac{G}{K}\right)$
defined by $\delta\left(f\right)=f*\overline{\delta}$. Then for every
$x\in\nicefrac{G}{K}$ 
\[
\delta\left(f\right)\left(x\right)=\sum\limits _{y\sim x}f\left(y\right).
\]
Namely, $\delta$ is nothing more then the adjacency operator (whose
name in the classical literature is Hecke operator).
\end{claim}
Let now $\Gamma$ be a cocompact discrete subgroup of $G=\mathrm{PGL}_{2}\left(F\right)$
(for simplicity assume also that $\Gamma$ is torsion free). Then
$\Gamma\backslash G/K$ is, on one hand a quotient of the $\left(q+1\right)$-regular
tree and, on the other hand, a quotient of the compact space $\Gamma\backslash G$.
Hence, this is nothing more than a finite $\left(q+1\right)$-regular
graph. Moreover, the discussion above shows that the spectral decomposition
of the adjacency matrix of this finite graph (and in particular its
eigenvalues) is intimately connected with the spectral decomposition
of $L^{2}\left(\Gamma\backslash G\right)$ as a unitary $G$-representation.
More precisely, in every irreducible $K$-spherical subrepresentation
$\rho$ of $L^{2}\left(\Gamma\backslash G\right)$, there is a $K$-invariant
function $f$, i.e.\ a function in $L^{2}\left(\Gamma\backslash G/K\right)$.
As explained above the one dimensional space spanned by $f$ is a
representation space $\overline{\rho}$ for $C$, which means that
$f$ is an eigenvector for the adjacency operator $\delta$ of the
finite graph $\Gamma\backslash G/K$. Moreover, every eigenvector
$f$ of $\delta$ in $L^{2}\left(\Gamma\backslash G/K\right)$ is
obtained like that (we can look at the $G$-subspace spanned by $f$,
thinking of it as a $K$-invariant function in $L^{2}\left(\Gamma\backslash G\right)$.)

The list of $K$-spherical irreducible unitary representations of
$\mathrm{PGL}_{2}\left(F\right)$ is well known (see \cite[Theorem 5.4.3]{lubotzky1994discrete}
and the references therein). There are representations of two kinds:
\begin{enumerate}
\item The tempered representations - these are the $K$-spherical irreducible
representations $\left(\mathcal{H},\rho\right)$ with the following
property: There exists $0\neq u,v\in\mathcal{H}$ such that $\phi:G\rightarrow\mathbb{C}$
defined by $\phi\left(g\right)=\left\langle \rho\left(g\right)u,v\right\rangle $
(the coefficient function of $\rho$ w.r.t.\ $u$ and $v$) is in
$L^{2+\varepsilon}\left(G\right)$ for every $\varepsilon>0$. The
$K$-spherical representations with this property are also called
in this case ``the principal series'' and they are characterized
by the property that the associated eigenvalue $\lambda$ of $\delta$
(as a generator of $C$ acting on the one dimensional space $\mathcal{H}^{K}$)
satisfies $\left|\lambda\right|\leq2\sqrt{q}$.
\item The non-tempered representations - these are the representations for
which the above $\lambda$ satisfies $2\sqrt{q}<\left|\lambda\right|\leq q+1$
.
\end{enumerate}
The above description explains why and how the representation of $G=\mathrm{PGL}_{2}\left(F\right)$
on $L^{2}\left(\Gamma\backslash G\right)$ is crucial for understanding
the combinatorics of the graph $\Gamma\backslash G/K$. In fact we
have (see \cite[Corollary 5.5.3]{lubotzky1994discrete}:
\begin{thm}
\label{thm:Ramanujan-rep}Let $\Gamma$ be a cocompact lattice in
$G=\mathrm{PGL}_{2}\left(F\right)$. Then $\Gamma\backslash G/K$
is a Ramanujan graph if and only if every irreducible $K$-spherical
$G$-subrepresentation of $L^{2}\left(\Gamma\backslash G\right)$
is tempered, with the exception of the trivial representation (which
corresponds to $\lambda=q+1$) and the possible exception of the sign
representation (the non trivial one dimensional representation $\mathrm{sg}:G\rightarrow\left\{ \pm1\right\} $)
which corresponds to $\lambda=-\left(q+1\right)$ and which appears
in $L^{2}\left(\Gamma\backslash G\right)$ iff $\Gamma\backslash G/K$
is bipartite.
\end{thm}
Proving that $\Gamma$'s as in the last theorem indeed exist is a
highly nontrivial issue which we discuss in the next section. This
will lead to (explicit) constructions of Ramanujan graphs.
\begin{rem}
In case $\Gamma$ is a non-uniform lattice in $G=\mathrm{PGL}_{2}\left(F\right)$
(which exists only if $\mathrm{char}\left(F\right)>0$) one can develop
also a theory of Ramanujan diagrams (cf.\ \cite{morgenstern1994ramanujan})
which is also of interest even for computer science (see \cite{morgenstern1995natural}).
\end{rem}

\subsection{\label{sub:Explicit-constructions}Explicit constructions}

In this section we will quote the deep results which imply that various
graphs are Ramanujan and then we will show how to use them to get
explicit constructions of such graphs.

Let $k$ be a global field, i.e.\ $k$ is a finite extension of $\mathbb{Q}$
or of $\mathbb{F}_{p}\left(t\right)$. Let $\mathcal{O}$ be the ring
of integers of $k$, $S$ a finite set of valuations of $k$ (containing
all the archimedean ones if $\mathrm{char}\left(k\right)=0$) and
$\mathcal{O}_{S}$ the ring of $S$-integers ($=\left\{ x\in k\,\middle|\,\nu\left(x\right)\geq0,\forall\nu\notin S\right\} $).
Let $\utilde{G}$ be a $k$-algebraic semisimple group with a fixed
embedding $\utilde{G}\hookrightarrow\mathrm{GL}_{m}$. A general result
asserts that
\[
\Gamma=\utilde{G}\left(\mathcal{O}_{S}\right):=\utilde{G}\left(k\right)\bigcap\mathrm{GL}_{m}\left(\mathcal{O}_{S}\right)
\]
is a lattice (= discrete subgroup of finite covolume) in $\prod\limits _{\nu\in S}\utilde{G}\left(k_{\nu}\right)$
where $k_{\nu}$ is the completion of $k$ w.r.t.\ the valuation
$\nu$. In some cases (few of these will be described below) $\utilde{G}\left(k_{\nu}\right)$
is compact for every $\nu\in S$ except of one $\nu_{0}\in S$. In
such a case the projection of $\Gamma$ to $\utilde{G}\left(k_{\nu_{0}}\right)$,
which is also denoted by $\Gamma$, is called \emph{an arithmetic
lattice in $\utilde{G}\left(k_{\nu_{0}}\right)$}.\textbf{ }The arithmetic
lattice $\Gamma$ comes with a system of congruence subgroups defined
for every $0\neq I\triangleleft\mathcal{O}_{S}$ as:
\[
\Gamma\left(I\right)=\Gamma\bigcap\ker\left(\mathrm{GL}_{m}\left(\mathcal{O}_{S}\right)\rightarrow\mathrm{GL}_{m}\left(\nicefrac{\mathcal{O}_{S}}{I}\right)\right).
\]
If $\utilde{G}\left(k_{\nu_{0}}\right)\simeq\mathrm{PGL}_{2}\left(F\right)$
(or more generally $\mathrm{PGL}_{d}\left(F\right)$ - see Chapter
\ref{sec:Ramanujan-Complexes}) where $F$ is a local field as in
§\ref{sub:Bruhat-Tits-trees}, we get the arithmetic groups we are
interested in. We can now state:
\begin{thm}
\label{thm:cong-temp}Let $\Gamma\left(I\right)\triangleleft\Gamma$
be a congruence subgroup of an arithmetic lattice $\Gamma$ of $G=\mathrm{PGL}_{2}\left(F\right)$
as above. Then every infinite dimensional $K$-spherical irreducible
subrepresentation of $L^{2}\left(\Gamma\left(I\right)\backslash G\right)$
is tempered.
\end{thm}
The only possible finite dimensional $K$-spherical representations
are the trivial one and the $\mathrm{sg}$ representation. From Theorem
\ref{thm:Ramanujan-rep} we can now deduce:
\begin{cor}
The graph $\Gamma\left(I\right)\backslash G/K$ is Ramanujan.
\end{cor}
Theorem \ref{thm:cong-temp} is a very deep result whose proof is
a corollary of various works by some of the greatest mathematicians
of the $20^{\mathrm{th}}$ century. It is based in particular on the
solution of the so called ``Peterson-Ramanujan conjecture''. (In
characteristic $0$, in two steps: by Eichler for weight two modular
forms which is the relevant case for us, and by Deligne in general.
In positive characteristic by Drinfeld). Then one needs to combine
it with the Jacquet\textendash{}Langlands correspondence. The reader
is referred to \cite{lubotzky1994discrete} for more and in particular
to the Appendix there by J.\ Rogawski which gives the general picture.

The last result give explicit graphs in the mathematical sense of
explicit, but it also paves the way for an explicit construction,
in the computer science sense, of Ramanujan graphs. We will present
the ones constructed in \cite{Lubotzky1988}.

When $G=\mathrm{PGL}_{2}\left(F\right)$, all the arithmetic lattices
in $G$ are obtained via quaternion algebras. Namely, let $D$ be
a quaternion algebra defined over $k$ and $\utilde{G}=\nicefrac{D^{\times}}{Z}$,
i.e.\ the invertible quaternions modulo the central ones. If $D$
splits over $\nu_{0}\in S$ (i.e.\ $D\otimes k_{\nu_{0}}\simeq M_{2}\left(k_{\nu_{0}}\right)$)
while it is ramified over all $\nu\in S\backslash\left\{ \nu_{0}\right\} $
(i.e. $D\otimes k_{\nu}$ is a division algebra in which case $\nicefrac{\left(D\otimes k_{\nu}\right)^{\times}}{Z\left(D\otimes k_{\nu}\right)}$
is a compact group) then $\utilde{G}\left(\mathcal{O}_{S}\right)$
gives rise to an arithmetic lattice in $\utilde{G}\left(k_{\nu_{0}}\right)=\mathrm{PGL}_{2}\left(k_{\nu_{0}}\right)$.
Such lattices (and their congruence subgroups) can be used for the
construction of arithmetic lattices.

Let us take a very concrete example: Let $D$ be the classical Hamilton
quaternion algebra; so $D$ is spanned over $\mathbb{Q}$ by $1,i,j$
and $k$ with $i^{2}=j^{2}=k^{2}=-1$ and $ij=-ji=k$. It is well
known that it is ramified over $\mathbb{R}$ and over $\mathbb{Q}_{2}$
while splits over $\mathbb{Q}_{p}$ for every odd prime $p$, and
that $\utilde{G}\left(\mathbb{R}\right)=\nicefrac{\mathbb{H}^{\times}}{\mathbb{R}^{\times}}\simeq\mathrm{SO}\left(3\right)$
while $\utilde{G}\left(\mathbb{Q}_{p}\right)=\nicefrac{M_{2}\left(\mathbb{Q}_{p}\right)^{\times}}{\mathbb{Q}_{p}^{\times}}\simeq\mathrm{PGL}_{2}\left(\mathbb{Q}_{p}\right)$.
Fix such a prime $p$ and set $S=\left\{ \nu_{p},\nu_{\infty}\right\} $.
Then $\mathcal{O}_{S}=\mathbb{Z}\left[\frac{1}{p}\right]=\left\{ \frac{a}{p^{n}}\,\middle|\, a\in\mathbb{Z},n\in\mathbb{N}\right\} $,
and as explained above, $D\left(\mathbb{Z}\left[\frac{1}{p}\right]\right)$
is a discrete subring of $D\left(\mathbb{R}\right)\times D\left(\mathbb{Q}_{p}\right)$,
while 
\[
\Gamma=\nicefrac{D\left(\mathbb{Z}\left[\frac{1}{p}\right]\right)^{\times}}{Z}\hookrightarrow\mathrm{SO}\left(3\right)\times\mathrm{PGL}_{2}\left(\mathbb{Q}_{p}\right)
\]
is a cocompact lattice. 

Moreover, Jacobi's classical theorem tells us that there are $8\left(p+1\right)$
solutions to the equation: $x_{0}^{2}+x_{1}^{2}+x_{2}^{2}+x_{3}^{2}=p$
with $\left(x_{0},x_{1},x_{2},x_{3}\right)\in\mathbb{Z}^{4}$. Assume
now $p\equiv1\left(\mathrm{mod}\:4\right)$. In this case three of
the $x_{i}$ are even and one is odd. If we agree to take those with
$x_{0}$ odd and positive we have a set $\Sigma$ of $p+1$ solutions
which come in pairs $\alpha_{1},\overline{\alpha_{1}},\ldots,\alpha_{s},\overline{\alpha_{s}}$
where $s=\frac{p+1}{2}$ and where we consider $\alpha$ as an integral
quaternion $\alpha=x_{0}+x_{1}i+x_{2}j+x_{3}k$, $\overline{\alpha}=x_{0}-x_{1}i-x_{2}j-x_{3}k$
so $\left\Vert \alpha_{i}\right\Vert =\left\Vert \overline{\alpha_{i}}\right\Vert =p$.
These $p+1$ elements give $p+1$ elements in $\utilde{G}\left(\mathbb{Q}_{p}\right)$.
Moreover, each $\alpha_{i}$ (and $\overline{\alpha_{i}}$) when considered
as an element of $\mathrm{PGL}_{2}\left(\mathbb{Q}_{p}\right)$ takes
the standard $\mathbb{Z}_{p}$-lattice in $\mathbb{Q}_{p}\times\mathbb{Q}_{p}$
(see §\ref{sub:Bruhat-Tits-trees}) to an immediate neighbor in the
tree and $\overline{\alpha_{i}}=\alpha_{i}^{-1}$. One can deduce
(see \cite[Corollary 2.1.11]{lubotzky1994discrete}) that the group
$\Lambda=\left\langle \alpha_{1},\overline{\alpha_{1}},\ldots,\alpha_{s},\overline{\alpha_{s}}\right\rangle $
is a free group on $s=\frac{p+1}{2}$ generators acting simply transitive
on the Bruhat-Tits $\left(q+1\right)$-regular tree $T$. One can
therefore identify this tree with the Cayley graph of $\Lambda$ with
respect to $\Sigma=\left\{ \alpha_{i}\,\middle|\, i=1,\ldots,s\right\} $.
As $\Lambda$ is also a lattice in $\mathrm{PGL}_{2}\left(\mathbb{Q}_{p}\right)$,
it is of finite index in $\Gamma$. One can check (using ``strong
approximation'' or directly) that if $q$ is another prime with $q\equiv1\left(\mathrm{mod}\:4\right)$
then $\Gamma\left(2q\right)\backslash T=\mathrm{Cay}\left(\nicefrac{\Lambda}{\Lambda\cap\Gamma\left(2q\right)};\Sigma\right)$.

Spelling out the meaning of this, one gets the following explicit
construction of Ramanujan graphs, which are usually referred to as
the LPS-graphs.
\begin{thm}[{\cite{Lubotzky1988}, see \cite[Theorem 7.4.3]{lubotzky1994discrete}}]
 Let $p$ and $q$ be different prime numbers with $p\equiv q\equiv1\left(\mathrm{mod}\:4\right)$.
Fix $\varepsilon\in\mathbb{F}_{q}$ satisfying $\varepsilon^{2}=-1$.
For each $\alpha_{i}=\left(x_{0},x_{1},x_{2},x_{3}\right)$, $i=1,\ldots,s$
in the set $\Sigma$ above, take the matrix 
\[
\widetilde{\alpha_{i}}=\left(\begin{matrix}x_{0}+\varepsilon x_{1} & x_{2}+\varepsilon x_{3}\\
-x_{2}+\varepsilon x_{3} & x_{0}-\varepsilon x_{1}
\end{matrix}\right)\in\mathrm{PGL}_{2}\left(\mathbb{F}_{q}\right)
\]
and $\widetilde{\Sigma}=\left\{ \widetilde{\alpha_{i}}\,\middle|\, i=1,\ldots,s\right\} $.
Let $H$ be the subgroup of $\mathrm{PGL}_{2}\left(\mathbb{F}_{q}\right)$
generated by $\widetilde{\Sigma}$ and $X^{p,q}=\mathrm{Cay}\left(H;\widetilde{\Sigma}\right)$,
the Cayley graph of $H$ with respect to $\widetilde{\Sigma}$. Then:
\begin{enumerate}
\item $X^{p,q}$ is a $\left(p+1\right)$-regular Ramanujan graph.
\item If $\left(\frac{p}{q}\right)=-1$, i.e.\ $p$ is not a quadratic
residue modulo $q$ then $H=\mathrm{PGL}_{2}\left(\mathbb{F}_{q}\right)$
and $X^{p,q}$ is a bi-partite graph, while if $\left(\frac{p}{q}\right)=1$,
$H=\mathrm{PSL}_{2}\left(\mathbb{F}_{q}\right)$ and $X^{p,q}$ is
not.
\end{enumerate}
\end{thm}
The main motivation for the construction of Ramanujan graphs has been
expanders, but the LPS graphs turned out to have various other remarkable
properties like high girth and high chromatic numbers (simultaneously!).
See (\cite{Lubotzky1988,lubotzky1994discrete,Sarnak1990,Valette1997})
for more.

In \cite{morgenstern1994existence}, Morgenstern constructed, for
every prime power $q$, infinitely many $\left(q+1\right)$-regular
Ramanujan graphs. This time by finding an arithmetic lattice in $\mathrm{PGL}_{2}\left(\mathbb{F}_{q}\left(\left(t\right)\right)\right)$
acting simply transitive on the Bruhat-Tits tree. Another such a construction
is given (somewhat hidden) in \cite{Lubotzky2005a} as a special case
of Ramanujan complexes (to be discussed in the next chapter). These
last mentioned Ramanujan graphs are also edge transitive and not merely
vertex transitive (see \cite{kaufman2011edge}).

\section{\label{sec:Ramanujan-Complexes}Ramanujan complexes}

This Chapter is devoted to the high-dimensional version of Ramanujan
graphs, the so called Ramanujan complexes. These are $\left(d-1\right)$-dimensional
simplicial complexes which are obtained as quotients of the Bruhat-Tits
building $\mathcal{B}_{d}$ associated with $\mathrm{PGL}_{d}\left(F\right)$,
$F$ a local field, just like the Ramanujan graphs were obtained as
quotients of the Bruhat-Tits tree of $\mathrm{PGL}_{2}\left(F\right)$.
What enables this, is the work of Lafforgue \cite{lafforgue2002chtoucas}
which extended to general $d$, the ``Ramanujan conjecture'' for
$\mathrm{GL}_{d}$ in positive characteristic, proved first by Drinfeld
\cite{drinfel1988proof} for $d=2$ (the work of Drinfeld was the
basis for the Ramanujan graph constructed by Morgenstern \cite{morgenstern1994existence}).
We will start in §\ref{sub:Bruhat-Tits-buildings} with the basic
definitions and results about buildings and will present the associated
Hecke operators, generalizing the Hecke operator (=adjacency matrix)
which appeared in Chapter 1. We will present the analogue of Alon-Boppana
Theorem and define Ramanujan complexes. In §\ref{sub:Representation-theory-of-PGLd}
we will survey shortly the representation theory of $\mathrm{PGL}_{d}\left(F\right)$
and just as in Theorem \ref{thm:Ramanujan-rep}, we will show that
representation theory is relevant for the combinatorics of $\Gamma\backslash\mathcal{B}_{d}$.
Then in §\ref{sub:Explicit-construction-complexes}, we will present
explicit constructions of Ramanujan complexes.

We will follow mainly \cite{Lubotzky2005a} and \cite{Lubotzky2005b},
where the reader can find precise references for the results mentioned
here. The reader is also referred to \cite{ballantine2000ramanujan,cartwright2003ramanujan,li2004ramanujan,sarveniazi2007explicit}
for related material.

\subsection{\label{sub:Bruhat-Tits-buildings}Bruhat-Tits buildings and Hecke
operators}

Let $K$ be any field. The spherical complex $\mathbb{P}^{d-1}\!\left(K\right)$
associated with $K^{d}$ is the simplicial complex whose vertices
are all the non-trivial (i.e.\ not $\left\{ 0\right\} $ and not
$K^{d}$) subspaces of $K^{d}$. Two subspaces $W_{1}$ and $W_{2}$
are on the same $1$-edge if either $W_{1}\subseteq W_{2}$ or $W_{2}\subseteq W_{1}$,
and $\left\{ W_{0},\ldots,W_{r}\right\} $ is an $r$-cell if every
pair $W_{i},W_{j}$ form an edge (i.e.\ $\mathbb{P}^{d-1}\!\left(K\right)$
is a ``clique complex''). It can be shown that this happens iff
after some reordering $W_{0}\subseteq W_{1}\subseteq\ldots\subseteq W_{r}$.
When $K=\mathbb{F}_{q}$, a finite field of order $q$, $\mathbb{P}^{1}\left(\mathbb{F}_{q}\right)$
is just a set of $q+1$ points, which can be identified with the projective
line over $\mathbb{F}_{q}$. For $d=3$, $\mathbb{P}^{2}\left(K\right)$
is the $\left(q+1\right)$-regular graph with $2\left(q^{2}+q+1\right)$
vertices of ``points'' versus ``lines'' of the projective plane.
In general, $\mathbb{P}^{d-1}\!\left(K\right)$ is a $\left(d-2\right)$-dimensional
simplicial complex.

We now describe $\mathcal{B}=\mathcal{B}_{d}\left(F\right)$, the
affine Bruhat-Tits building of type $\widetilde{A}_{d-1}$ associated
with $F$. Here $F$ is a local field with $\mathcal{O}$, $\pi$
and $\mathfrak{m}$ as in §\ref{sub:Bruhat-Tits-trees} and $\nicefrac{\mathcal{O}}{\mathfrak{m}}=\mathbb{F}_{q}$.
An $\mathcal{O}$-lattice in $V=F^{d}$ is a finitely generated $\mathcal{O}$-submodule
$L$ of $V$ such that $L$ contains an $F$-basis of $V$. Two lattices
$L_{1}$ and $L_{2}$ are equivalent if $L_{1}=\lambda L_{2}$ for
some $0\neq\lambda\in F$. The vertices of $\mathcal{B}$ are the
equivalence classes of $\mathcal{O}$-lattices in $V$, and two distinct
equivalent classes $\left[L_{1}\right]$ and $\left[L_{2}\right]$
are adjacent in $\mathcal{B}$ if there exist representatives $L_{1}'\in\left[L_{1}\right]$
and $L_{2}'\in\left[L_{2}\right]$ s.t.\ $\pi L_{1}'\subseteq L_{2}'\subseteq L_{1}'$.
The $r$-simplices of $\mathcal{B}$ ($r\geq2$) are the subsets $\left\{ \left[L_{0}\right],\ldots,\left[L_{r}\right]\right\} $
such that all pairs $\left[L_{i}\right]$ and $\left[L_{j}\right]$
are adjacent. It can be shown that $\left\{ \left[L_{0}\right],\ldots,\left[L_{r}\right]\right\} $
forms a simplex if and only if there exist representatives $L_{i}'\in\left[L_{i}\right]$
such that after reordering the indices, $\pi L_{r}'\subseteq L_{0}'\subseteq\ldots\subseteq L_{r}'$.
The complex $\mathcal{B}$ is therefore of dimension $d-1=\mathrm{rank}_{F}\left(\mathrm{PGL}_{d}\left(F\right)\right)$.
This is a special case of the Bruhat-Tits building associated with
a reductive group over $F$. The next theorem is also a special case
which generalizes Theorem \ref{thm:tree-is-regular}:
\begin{thm}
The complex $\mathcal{B}_{d}\left(F\right)$ is contractible. The
link of each vertex is isomorphic to $\mathbb{P}^{d-1}\left(\mathbb{F}_{q}\right)$.
\end{thm}
The vertices of $\mathcal{B}$ come with a natural coloring (``type'').
Let $\tau:\mathcal{B}^{0}\rightarrow\nicefrac{\mathbb{Z}}{d\mathbb{Z}}$
be defined as follows: Let $\mathcal{O}^{d}\subseteq V$ be the standard
lattice in $V$. For any lattice $L$, there exists $g\in GL\left(V\right)=GL_{d}\left(F\right)$
such that $L=g\left(\mathcal{O}^{d}\right)$. Define $\tau\left(\left[L\right]\right)=\nu\left(\det\left(g\right)\right)\left(\mathrm{mod}\: d\right)$
where $\nu$ is the valuation of $F$, e.g.\ for $F=\mathbb{F}_{q}\left(\left(t\right)\right)$,
$\nu\left(\sum_{i=m}^{\infty}a_{i}t^{i}\right)=m$ when $m\in\mathbb{Z}$
and $a_{m}\neq0$.

The group $\mathrm{GL}_{d}\left(F\right)$ acts transitively on the
$\mathcal{O}$-lattices in $V$ and its center preserves the equivalence
classes. As the action preserves inclusions, $G=\mathrm{PGL}_{d}\left(F\right)$
acts on the building $\mathcal{B}$. It acts transitively on $\mathcal{B}^{0}$
- the vertices - without preserving their colors, but $\mathrm{PSL}_{d}\left(F\right)$
does preserves them. The stabilizer of the (equivalence class of the)
standard lattice is $K=\mathrm{PGL}_{d}\left(\mathcal{O}\right)$.
Hence $\mathcal{B}^{0}$ can be identified with $\nicefrac{G}{K}$.
To every directed edge $\left(x,y\right)\in\mathcal{B}^{1}$ one can
associate the color $\tau\left(y\right)-\tau\left(x\right)\in\nicefrac{\mathbb{Z}}{d\mathbb{Z}}$.
The color of edges is preserved by $\mathrm{PGL}_{d}\left(F\right)$.

Let us now define $d-1$ operators - the Hecke operators - as follows:
For $1\leq k\leq d-1$, $f\in L^{2}\left(\mathcal{B}^{0}\right)$
and $x\in\mathcal{B}^{0}$,
\begin{equation}
\left(A_{k}f\right)\left(x\right)=\sum f\left(y\right)\label{eq:Hecke-Ak}
\end{equation}
where the summation is over the neighbors $y$ of $x$ such that $\tau\left(y\right)-\tau\left(x\right)=k\in\nicefrac{\mathbb{Z}}{d\mathbb{Z}}$.
This really amounts to a sum over the sublattices of $L$ containing
$\pi L$, whose projection in $\nicefrac{L}{\pi L}$ is of codimension
$k$ there. Note that $A_{k}$ commutes with the action of $\mathrm{PGL}_{d}\left(F\right)$.
One can show that these operators are bounded, normal and commute
with each other. But in general they are not self-adjoint. In fact,
$A_{k}^{*}=A_{d-k}$ so $A_{k}+A_{d-k}$ is self-adjoint. Moreover
$\Delta=\sum_{k=1}^{d-1}A_{k}$ is the adjacency operator of the $1$-skeleton
of $\mathcal{B}$. For $d=2$, we only have $A_{1}=A_{1}^{*}$ which
is indeed the adjacency operator of the tree. As the operators $A_{k}$
commute with each other they can be diagonalized simultaneously. Their
common spectrum is therefore a subset $\Sigma_{d}$ of $\mathbb{C}^{d-1}$.
\begin{thm}
\label{thm:hecke-spectrum}Let $S=\left\{ \left(z_{1},\ldots,z_{d}\right)\in\mathbb{C}^{d}\,\middle|\,\left|z_{i}\right|=1\: and\:\prod_{i=1}^{d}z_{i}=1\right\} $
and $\sigma:\left(z_{1},\ldots,z_{d}\right)\mapsto\left(\lambda_{1},\ldots,\lambda_{d-1}\right)$
where $\lambda_{k}=q^{\frac{k\left(d-k\right)}{2}}\sigma_{k}\left(z_{1},\ldots,z_{d}\right)$.
Here $\sigma_{k}$ is the $k^{\mathrm{th}}$ elementary symmetric
function, i.e.\ $\sigma_{k}\left(z_{1},\ldots,z_{d}\right)=\sum\limits _{i_{i}<\ldots<i_{k}}z_{i_{1}}\cdot\ldots\cdot z_{i_{k}}$.
Then $\sigma\left(S\right)$ is equal to $\Sigma_{d}$, the simultaneous
spectrum of $A_{1},\ldots,A_{d-1}$ acting on $L^{2}\left(\mathcal{B}^{0}\right)$.
\end{thm}
Note that indeed $\lambda_{k}=\overline{\lambda_{d-k}}$ as had to
be expected, since $A_{k}=A_{d-k}^{*}$. Also for $d=2$, 
\[
\Sigma_{2}=\sigma\left(S\right)=\left\{ q^{\frac{1}{2}}\left(z+\frac{1}{z}\right)\,\middle|\, z\in\mathbb{C},\left|z\right|=1\right\} =\left[-2\sqrt{q},2\sqrt{q}\right]
\]
which shows that Theorem \ref{thm:hecke-spectrum} is a generalization
of Proposition \ref{prop:tree-spec}.

Ramanujan $\left(q+1\right)$-regular graphs were defined as the finite
quotients of $\mathcal{B}_{2}=T_{q+1}$ whose ``non-trivial'' eigenvalues
are all in $\Sigma_{2}$. Similarly we will define Ramanujan complexes
as quotients of $\mathcal{B}_{d}$ whose ``non-trivial'' eigenvalues
are in $\Sigma_{d}$. Let us describe first the trivial eigenvalues:
Recall that for $d=2$ we have two such: $\left(q+1\right)$ and $-\left(q+1\right)$.
They appear in all the finite quotients $\Gamma\backslash\mathcal{B}_{2}$
when $\Gamma$ preserves the colors of the vertices (and only $q+1$
appears in all the finite quotients).

So, assume $\Gamma\leq\mathrm{PGL}_{d}\left(F\right)$ is a cocompact
lattice preserving the colors of the vertices of $\mathcal{B}^{0}$.
So, $\tau$ is well defined on $X=\Gamma\backslash\mathcal{B}^{0}$.
For a $d^{\mathrm{th}}$ root of unity $\xi$, define $f_{\xi}:X\rightarrow\mathbb{C}$
by $f_{\xi}\left(x\right)=\xi^{\tau\left(x\right)}$. Now, $A_{k}f_{\xi}\left(x\right)$
sums over the neighbors of $x$ of color $\tau\left(x\right)+k\left(\mathrm{mod}\: d\right)$
and there are $\left[\begin{smallmatrix}d\\
k
\end{smallmatrix}\right]_{q}$ vertices like that (where $\left[\begin{smallmatrix}d\\
k
\end{smallmatrix}\right]_{q}$ denotes the number of subspaces of codimension $k$ in $\mathbb{F}_{q}^{d}$).
Hence $A_{k}f_{\xi}\left(x\right)=\left[\begin{smallmatrix}d\\
k
\end{smallmatrix}\right]_{q}\xi^{\tau\left(x\right)+k}=\left[\begin{smallmatrix}d\\
k
\end{smallmatrix}\right]_{q}\xi^{k}f_{\xi}\left(x\right)$. Thus, for every $\xi\in\mathbb{C}$ with $\xi^{d}=1$, we get a
simultaneous ``trivial'' eigenvalue $\left(\left[\begin{smallmatrix}d\\
1
\end{smallmatrix}\right]_{q}\xi^{1},\ldots,\left[\begin{smallmatrix}d\\
k
\end{smallmatrix}\right]_{q}\xi^{k},\ldots,\left[\begin{smallmatrix}d\\
d-1
\end{smallmatrix}\right]_{q}\xi^{d-1}\right)$. These are the $d$ trivial eigenvalues. Again, for $d=2$, we get
$\left[\begin{smallmatrix}2\\
1
\end{smallmatrix}\right]_{q}\cdot1=q+1$ and $\left[\begin{smallmatrix}2\\
1
\end{smallmatrix}\right]_{q}\left(-1\right)=-\left(q+1\right)$. We can now define
\begin{defn}
A Ramanujan complex (of type $\widetilde{A}_{d-1}$) is a finite quotient
$X=\Gamma\backslash\mathcal{B}_{d}$ satisfying: every simultaneous
eigenvalue $\left(\lambda_{1},\ldots,\lambda_{k},\ldots,\lambda_{d-1}\right)$
of $\left(A_{1},\ldots,A_{k},\ldots,A_{d-1}\right)$ satisfies: either
$\left(\lambda_{1},\ldots,\lambda_{d-1}\right)$ is one of the $d$
trivial eigenvalues (i.e.\ $\left(\lambda_{1},\ldots,\lambda_{d-1}\right)=\left(\left[\begin{smallmatrix}d\\
1
\end{smallmatrix}\right]_{q}\xi^{1},\ldots,\left[\begin{smallmatrix}d\\
d-1
\end{smallmatrix}\right]_{q}\xi^{d-1}\right)$ for some $d^{\mathrm{th}}$ root of unity $\xi$) or $\left(\lambda_{1},\ldots,\lambda_{d-1}\right)\in\Sigma_{d}$,
described in Theorem \ref{thm:hecke-spectrum}.
\end{defn}
In the case of $d=2$, we saw the Alon-Boppana Theorem (Theorem \ref{thm:Alon-Boppana})
which shows that the Ramanujan bounds are the strongest one can hope
from an infinite family of $\left(q+1\right)$-regular graphs (for
a fixed $q$). The following theorem is a strong high dimensional
version.
\begin{thm}[{\cite[Theorem 4.3]{li2004ramanujan}}]
\label{thm:Li-AlonBoppana}Let $X_{i}$ be a family of finite quotients
of $\mathcal{B}_{d}$ with unbounded injective radius (recall that
the injective radius of a quotient $\pi:\mathcal{B}\rightarrow\Gamma\backslash\mathcal{B}$
is the maximal $r$ such that $\pi$ is an isomorphism when restricted
to any ball of radius $r$ in $\mathcal{B}$). Then $\overline{\bigcup\mathrm{spec}_{X_{i}}\left(A_{1},\ldots,A_{d-1}\right)}\supseteq\Sigma_{d}$.
\end{thm}
This shows that the best we can hope for the $X_{i}$'s is to be Ramanujan.
Note that $\mathrm{spec}_{X_{i}}\left(A_{1},\ldots,A_{d-1}\right)$
is a finite set for every $i$.

Let us end this section with the following remark:
\begin{rem}
The trivial eigenvalues of $\left(A_{1},\ldots,A_{d-1}\right)$ are
$\left(\lambda_{1},\ldots,\lambda_{d-1}\right)=\left(\left[\begin{smallmatrix}d\\
1
\end{smallmatrix}\right]_{q}\xi^{1},\ldots,\left[\begin{smallmatrix}d\\
d-1
\end{smallmatrix}\right]_{q}\xi^{d-1}\right)$. So for $A_{k}$, $\left|\lambda_{k}\right|=\left[\begin{smallmatrix}d\\
k
\end{smallmatrix}\right]_{q}\approx q^{k\left(d-k\right)}$ while the Ramanujan bound gives:
\[
\left|\lambda_{k}\right|\leq q^{\frac{k\left(d-k\right)}{2}}\left|\sigma_{k}\left(z_{1},\ldots,z_{k}\right)\right|\leq{d \choose k}q^{\frac{k\left(d-k\right)}{2}}
\]
so for $d$ fixed and $q$ large, the Ramanujan bound is approximately
the square root of the trivial eigenvalue.
\end{rem}
In §\ref{sub:Eigenvalues-and-expanders} we mentioned that Ramanujan
graphs can bee characterized by the fact that their zeta functions
satisfy the Riemann hypothesis. Recently there have been some efforts
to associate zeta functions to higher dimensional complexes with the
hope to give a similar characterization for Ramanujan complexes of
dimension $2$. See \cite{deitmar2006ihara,storm2006zeta,kang2010zeta}.
It will be nice if this theory could be extended also to higher dimensions.

\subsection{\label{sub:Representation-theory-of-PGLd}Representation theory of
$\mathrm{PGL}_{d}$}

In this section we will describe some basic results from the representation
theory of $\mathrm{PGL}_{d}\left(F\right)$, $F$ a local field. For
a more comprehensive survey see \cite{cartier1979representations}.
We will give only those results which are needed for our combinatorial
application. The goal is to get a high dimensional generalization
of Theorem \ref{thm:Ramanujan-rep}, i.e., a representation theoretic
formulation of Ramanujan complexes.

Let $G=\mathrm{PGL}_{d}\left(F\right)$ and $K=\mathrm{PGL}_{d}\left(\mathcal{O}\right)$,
$\mathcal{O}$ the ring of integers of $F$. An irreducible unitary
representation $\left(\mathcal{H},\rho\right)$ of $G$ is called
$K$-spherical if the space of $K$-fixed points $\mathcal{H}^{K}$
is non-zero. In this case $\dim\mathcal{H}^{K}=1$. Let $C=C_{c}\left(K\backslash G/K\right)$
be the algebra of compactly supported bi-$K$-invariant functions
from $G$ to $\mathbb{C}$, with multiplication defined by convolution
\[
f_{1}*f_{2}\left(x\right)=\int_{G}f_{1}\left(xg\right)f_{2}\left(g^{-1}\right)dg.
\]
The algebra $C$ is called the Hecke algebra of $G$. Let $\widetilde{\pi}_{k}=\mathrm{diag}\left(\pi,\pi,\ldots,\pi,1,1,\ldots,1\right)\in\mathrm{GL}_{d}\left(F\right)$
with $\det\left(\widetilde{\pi}_{k}\right)=\pi^{d-k}$, where $\pi$
is the uniformizer of $F$. Denote by $\pi_{k}$ the image of $\widetilde{\pi}_{k}$
in $\mathrm{PGL}_{d}\left(F\right)$ and let $A_{k}$ be the characteristic
function of $K\pi_{k}K$. Clearly $\left\{ A_{k}\right\} _{k=1}^{d-1}\subseteq C$
(note $\pi_{0}=\pi_{d}=I_{d}$). Less trivial is the fact that $C$
is commutative and is freely generated as a commutative algebra by
$A_{1},\ldots,A_{d-1}$ (cf.\ \cite[Chap.\ $V$]{macdonald1979symmetric}).
Every irreducible unitary representation $\left(\mathcal{H},\rho\right)$
of $G$ gives rise to a representation of $C$ on $\mathcal{H}^{K}$
and when $\mathcal{H}^{K}\neq\left\{ 0\right\} $, this last representation
is in fact given by a homomorphism $w:C\rightarrow\mathbb{C}$, $f\cdot v_{0}=w\left(f\right)v_{0}$
for $f\in C$. The representation $\rho$ is uniquely determined by
$w$ (cf.\ \cite[Prop.\ 2.2]{Lubotzky2005a} ) and $w$ is determined
by the $\left(d-1\right)$-tuple $\left(w\left(A_{1}\right),\ldots,w\left(A_{d-1}\right)\right)\in\mathbb{C}^{d-1}$.

Let us put this in a somewhat more known formulation: a more common
parametrization of the irreducible spherical representations of $\mathrm{GL}_{d}\left(F\right)$
(and hence also of $\mathrm{PGL}_{d}\left(F\right)$) is by their
\emph{Satake parametrization} $\left(z_{1},\ldots,z_{d}\right)\in\nicefrac{\left(\mathbb{C}^{\times}\right)^{d}}{Sym\left(d\right)}$.
This parametrization is related but not the same as the one we discuss
here. Let us just mention here that
\begin{enumerate}
\item A representation of $\mathrm{GL}_{d}\left(F\right)$ with Satake parameters
$\left(z_{1},\ldots,z_{k}\right)$ factors through $\mathrm{PGL}_{d}\left(F\right)$
iff $\prod_{i=1}^{d}z_{i}=1$.
\item If $\left(\mathcal{H},\rho\right)$ is an irreducible spherical representation
of $\mathrm{PGL}_{d}\left(F\right)$ with Satake parameters $\left(z_{1},\ldots,z_{d}\right)$
then $w\left(A_{k}\right)$ in the notation above is given by $w\left(A_{k}\right)=q^{\frac{k\left(d-k\right)}{2}}\sigma_{k}\left(z_{1},\ldots,z_{d}\right)$
where $\sigma_{k}$ is the $k^{\mathrm{th}}$ elementary symmetric
function on $d$ variables, $\sigma_{k}\left(z_{1},\ldots,z_{d}\right)=\sum\limits _{i_{i}<\ldots<i_{k}}z_{i_{1}}\cdot\ldots\cdot z_{i_{k}}$.
\item An irreducible representation $\left(\mathcal{H},\rho\right)$ is
called \emph{tempered}, if there exists $0\neq v,u\in\mathcal{H}$
such that the coefficient function $\psi\left(g\right)=\left\langle \rho\left(g\right)v,u\right\rangle $
is in $L^{2+\varepsilon}\left(G\right)$ for every $\varepsilon>0$.
These are exactly the representations which are \emph{weakly contained}
(in the sense of the Fell topology) in the representation of $G$
on $L^{2}\left(G\right)$. If such a representation is also $K$-spherical
then it is weakly contained in the representation of $G$ on $L^{2}\left(\nicefrac{G}{K}\right)=L^{2}\left(\mathcal{B}_{0}\right)$.
In terms of Satake parameters, $\rho$ is tempered iff $\left|z_{i}\right|=1$
for all $i$. 
\end{enumerate}
The reader is referred to more information in \cite{Lubotzky2005a}
and for the general theory in \cite{cartier1979representations}.
At this point, especially in light of (b) and (c) the reader may start
to guess the connection to Ramanujan complexes. Let us spell it out
explicitly.

Let $L_{0}=\mathcal{O}e_{1}+\ldots+\mathcal{O}e_{d}$ be the standard
$\mathcal{O}$-lattice in $V=F^{d}$ and $\left[L_{0}\right]$ its
equivalence class, which corresponds to $K$ under the identification
$\nicefrac{G}{K}=\mathcal{B}^{0}$. Let $\Omega_{k}$ be the set of
neighbors of color $k$ of $\left[L_{0}\right]$. Then $\pi_{k}^{-1}K\in\nicefrac{G}{K}=\mathcal{B}^{0}$
is one of these neighbors and $K$ (as a subgroup of $G$) acts transitively
on $\Omega_{k}$ so that $K\pi_{k}^{-1}K=\bigcup yK$ where the union
is over all $yK\in\Omega_{k}$. Multiplying from the left by an arbitrary
$g\in G$, we see that the neighbors of the vertex $gK$ forming an
edge of color $k$ with it, are exactly $\left\{ gyK\right\} _{yK\in\Omega_{k}}$.
It follows that the operator $A_{k}$ defined in \eqref{eq:Hecke-Ak}
in §\ref{sub:Bruhat-Tits-buildings}, can be expressed as follows:
Identifying $L^{2}\left(\mathcal{B}^{0}\right)=L^{2}\left(\nicefrac{G}{K}\right)$
with the right $K$-invariant functions in $L^{2}\left(G\right)$,
and assuming that $K$ has Haar measure one, for $f\in L^{2}\left(\mathcal{B}^{0}\right)$,
and $gK\in\mathcal{B}^{0}$

\begin{equation}
\begin{gathered}\left(A_{k}f\right)\left(gK\right)=\sum_{yK\in\Omega_{k}}f\left(gyK\right)=\sum_{yK\in\Omega_{k}}\int_{yK}f\left(gz\right)dz\\
=\int_{K\pi_{k}^{-1}K}f\left(gz\right)dz=\int_{G}f\left(gz\right)\mathbf{1}_{K\pi_{k}K}\left(z^{-1}\right)dz=\left(f*A_{k}\right)\left(gK\right)
\end{gathered}
\label{eq:Hecke-conv}
\end{equation}
where $A_{k}$ at the right hand side of equation \eqref{eq:Hecke-conv}
is the characteristic function of $K\pi_{k}K$, as defined in this
section. No confusion should occur here as eq.\ \eqref{eq:Hecke-conv}
shows that the Hecke operators of §\ref{sub:Bruhat-Tits-buildings}
and the Hecke operators of §\ref{sub:Representation-theory-of-PGLd}
are essentially the same thing! When $C=C_{c}\left(K\backslash G/K\right)$
acts on $L^{2}\left(\nicefrac{G}{K}\right)$, $A_{k}$ acts as the
adjacency operators summing over all the neighbors with edges of color
$k$.

We can now use this to deduce the main goal of this subsection (see
\cite[Prop.\ 1.5]{Lubotzky2005a})
\begin{prop}
\label{prop:ramanujan-tempered}Let $\Gamma$ be a cocompact lattice
of $\mathrm{PGL}_{d}\left(F\right)$. Then $\Gamma\backslash\mathcal{B}$
is a Ramanujan complex if and only if every irreducible spherical
infinite dimensional $G$-subrepresentation of $L^{2}\left(\Gamma\backslash\mathrm{PGL}_{d}\left(F\right)\right)$
is tempered.\end{prop}
\begin{proof}[Sketch of proof]
Assume every irreducible spherical infinite dimensional subrepresentation
of $\mathcal{H}=L^{2}\left(\Gamma\backslash\mathrm{PGL}_{d}\left(F\right)\right)$
is tempered. As $\Gamma$ is cocompact, $\mathcal{H}$ is a direct
sum of irreducible representations. Let $f\in L^{2}\left(\Gamma\backslash G/K\right)$
be a non-trivial simultaneous eigenfunction of the Hecke operators
$A_{k}$ with $A_{k}f=\lambda_{k}f$. As $\mathrm{PSL}_{d}\left(F\right)$
has no nontrivial finite dimensional representations, every finite
dimensional representation of $\mathrm{PGL}_{d}\left(F\right)$ factors
through $\nicefrac{\mathrm{PGL}_{d}\left(F\right)}{\mathrm{PSL}_{d}\left(F\right)}\simeq\nicefrac{F^{\times}}{\left(F^{\times}\right)^{d}}$.
Since $F^{\times}\simeq\mathbb{Z}\times\mathcal{O}^{\times}$, we
have $\nicefrac{F^{\times}}{\mathcal{O}^{\times}\left(F^{\times}\right)^{d}}\simeq\nicefrac{\mathbb{Z}}{d\mathbb{Z}}$
and since $f$ is fixed by $K$, if $f$ lies in a finite dimensional
$G$-subspace, it correspond to one of the $d$ trivial eigenvalue.
If $f$ spans an infinite dimensional $G$-space, then it is tempered,
its Satake parameters $\left(z_{1},\ldots,z_{d}\right)$ satisfy $\prod z_{i}=1$
and $\left|z_{i}\right|=1$. The corresponding eigenvalues of $A_{k}$
are, as explained in point (b) above, in $\Sigma_{d}$ as defined
in §\ref{sub:Bruhat-Tits-buildings}.

In the other direction: If $\mathcal{H}_{1}$ is an irreducible spherical
infinite dimensional subrepresentation of $L^{2}\left(\Gamma\backslash G\right)$,
then its unique (up to scalar) $K$-fixed vector $f$ is a simultaneous
eigenvector of all the $A_{k}$'s where $A_{k}f=\lambda_{k}f$. By
assumption $\left(\lambda_{1},\ldots,\lambda_{d-1}\right)\in\Sigma_{d}$,
from which we deduce that the Satake parameters $z_{i}$ all satisfy
$\left|z_{i}\right|=1$ and the representation is tempered.
\end{proof}
So, once again, as we saw for Ramanujan graphs, the problem of constructing
Ramanujan complexes moves from combinatorics to representation theory.
In the next subsection, we will describe how deep results in the area
of automorphic forms lead to such combinatorial constructions.

\subsection{\label{sub:Explicit-construction-complexes}Explicit construction
of Ramanujan complexes}

We will start with a general result which gives a lot of Ramanujan
complexes. We then continue to present an explicit construction.

Let us first recall some notations and add a few more: Let $k$ be
a global field of characteristic $p>0$ and $D$ a division algebra
of degree $d$ over $k$. Denote by $\utilde{G}$ the $k$-algebraic
group $\nicefrac{D^{\times}}{k^{\times}}$, and fix an embedding of
$\utilde{G}$ into $\mathrm{GL}_{n}$ for some $n$. Let $T$ be the
finite set of valuations of $k$ for which $D$ does not split. We
assume that for every $\nu\in T$, $D_{\nu}=D\otimes_{k}k_{\nu}$
is a division algebra. Let $\nu_{0}$ be a valuation of $k$ which
is not in $T$ and $F=k_{\nu_{0}}$, so that $\utilde{G}\left(F\right)\simeq\mathrm{PGL}_{d}\left(F\right)$,
and denote $S=T\bigcup\left\{ \nu_{0}\right\} $. For $\mathcal{O}_{S}=\left\{ x\in k\,\middle|\,\nu\left(x\right)\geq0\;\forall\nu\notin S\right\} $
the ring of $S$-integers in $k$, $\utilde{G}\left(\mathcal{O}_{S}\right):=\utilde{G}\left(k\right)\bigcap\mathrm{GL}_{n}\left(\mathcal{O}_{S}\right)$
embeds diagonally as a discrete subgroup of $\prod_{\nu\in S}\utilde{G}\left(k_{\nu}\right)$.
As $\utilde{G}\left(k_{\nu}\right)$ is compact for $\nu\in T$, projecting
$\utilde{G}\left(\mathcal{O}_{S}\right)$ into $\utilde{G}\left(k_{\nu_{0}}\right)=\utilde{G}\left(F\right)\simeq\mathrm{PGL}_{d}\left(F\right)$
gives an embedding of $\utilde{G}\left(\mathcal{O}_{S}\right)$ as
a discrete subgroup in $\mathrm{PGL}_{d}\left(F\right)$, which we
denote by $\Gamma$. In fact, by a general result on arithmetic subgroups,
$\Gamma$ is a cocompact lattice in $\mathrm{PGL}_{d}\left(F\right)$.
Thus if $\mathcal{B}=\mathcal{B}_{d}\left(F\right)$ is the Bruhat-Tits
building associated with $\mathrm{PGL}_{d}\left(F\right)$, then $\Gamma\backslash\mathcal{B}$
is a finite complex. The same is true when we mod $\mathcal{B}$ by
any finite index subgroup of $\Gamma$. In particular, if $0\neq I\triangleleft\mathcal{O}_{S}$
is an ideal, then the congruence subgroup $\Gamma\left(I\right):=\ker\left(\utilde{G}\left(\mathcal{O}_{S}\right)\rightarrow\utilde{G}\left(\nicefrac{\mathcal{O}_{S}}{I}\right)\right)$
is of finite index in $\Gamma$ and $\Gamma\left(I\right)\backslash\mathcal{B}$
is a finite simplicial complex covered by $\mathcal{B}$.
\begin{thm}
\label{thm:congruence-ramanujan-complex}For $\Gamma$ and $I$ as
above, $\Gamma\left(I\right)\backslash\mathcal{B}$ is a Ramanujan
complex.
\end{thm}
A word of warning: if $d$ is not a prime then there are ideals in
$\mathcal{O}_{\nu_{0}}=\left\{ x\in k\,\middle|\,\nu\left(x\right)\geq0\;\forall\nu\neq\nu_{0}\right\} $
(so they may disappear in $\mathcal{O}_{S}$!) which give non-Ramanujan
complexes. We refer to \cite{Lubotzky2005a} for this delicate point
as well as for a proof of Theorem \ref{thm:congruence-ramanujan-complex}.
We will not try to explain the proof, but rather give few hints about
it. The Theorem is proved there by going from local to global. By
Proposition \ref{prop:ramanujan-tempered}\textbf{ }above, $\Gamma\left(I\right)\backslash\mathcal{B}$
is Ramanujan iff every infinite dimensional irreducible spherical
subrepresentation $\rho_{0}$ of $L^{2}\left(\Gamma\left(I\right)\backslash\mathrm{PGL}_{d}\left(F\right)\right)$
is tempered. One shows that such $\rho_{0}$ is a local factor at
$\nu_{0}$ of an automorphic adelic subrepresentation $\rho'$ of
$L^{2}\left(\utilde{G}\left(k\right)\backslash\utilde{G}\left(\mathbb{A}\right)\right)$
where $\mathbb{A}$ is the ring of adeles of $k$. By using the Jacquet\textendash{}Langlands
correspondence, one can replace $\rho'$ by a suitable subrepresentation
$\rho$ of $L^{2}\left(\mathrm{PGL}_{d}\left(k\right)\backslash\mathrm{PGL}_{d}\left(\mathbb{A}\right)\right)$.
Then one appeals to the work of Lafforgue \cite{lafforgue2002chtoucas}
(for which he got the Fields medal!) which is an extension to general
$d$ of the ``Ramanujan conjecture'' proved by Drinfeld for $d=2$.
This last result says that for various adelic automorphic representations,
the local factors are tempered. This can be applied to $\rho$ to
deduce that our $\rho_{0}$ is tempered and hence $\Gamma\left(I\right)\backslash\mathcal{B}$
is Ramanujan.

The description of the complexes we gave is pretty abstract but it
can be made very explicit in some cases. To this end we will make
use (following \cite{Lubotzky2005b}) of a remarkable arithmetic lattice
$\Gamma$ constructed by Cartwright and Steger \cite{cartwright1998family}.
This lattice has the following amazing\emph{ }property: It acts simply
transitively on the vertices of the building $\mathcal{B}_{d}$. Such
lattices are rare; for example in characteristic zero such lattices
exist only for finitely many $d$'s (see \cite{mohammadi2012discrete})).
Let us describe their (somewhat technical) construction:

We start with the global field $k=\mathbb{F}_{q}\left(y\right)$,
whose valuations are $\nu_{g}$ for every irreducible polynomial $g$
in $\mathbb{F}_{q}\left[y\right]$, and the \emph{minus degree }valuation,
$\nu_{\frac{1}{y}}\left(\nicefrac{f}{g}\right)=\deg g-\deg f$. Let
$\mathbb{F}_{q^{d}}$ be the field extension of $\mathbb{F}_{q}$
of degree $d$ and $\phi$ a generator of the Galois group $\mathrm{Gal}\left(\nicefrac{\mathbb{F}_{q^{d}}}{\mathbb{F}_{q}}\right)\simeq\nicefrac{\mathbb{Z}}{d\mathbb{Z}}$.
Fix a basis $\xi_{0},\ldots,\xi_{d-1}$ of $\mathbb{F}_{q^{d}}$ over
$\mathbb{F}_{q}$ with $\xi_{i}=\phi^{i}\left(\xi_{0}\right)$. Let
$D$ be the $k$-algebra with basis $\left\{ \xi_{i}z^{j}\right\} _{i,j=0}^{d-1}$
and relations $z\xi_{i}=\phi\left(\xi_{i}\right)z$ and $z^{d}=1+y$.
Then $D$ is a division algebra which ramifies at $T=\left\{ \nu_{1+y},\nu_{\frac{1}{y}}\right\} $
and splits at all other completions of $k$ (see \cite[Prop.\ 3.1]{Lubotzky2005b}).
That is, $D_{\nu_{1+y}}=D\otimes_{k}k_{\nu_{1+y}}=D\otimes_{k}\mathbb{F}_{q}\left(\left(1+y\right)\right)$
and $D_{\nu_{\frac{1}{y}}}={D\otimes_{k}\mathbb{F}_{q}\left(\left(\frac{1}{y}\right)\right)}$
are division algebras, while $D_{\nu}\simeq M_{d}\left(k_{\nu}\right)$
for $\nu\notin T$. In particular, $\utilde{G}\left(k_{\nu}\right)\simeq\mathrm{PGL}_{d}\left(k_{\nu}\right)$
for $\nu\notin T$, where we recall that  $\utilde{G}$ denotes the
$k$-algebraic group $\nicefrac{D^{\times}}{k^{\times}}$.

For $\nu_{0}$ we take the valuation $\nu_{y}$, which is given explicitly
by $\nu_{y}\left(a_{m}y^{m}+\ldots+a_{n}y^{n}\right)=m$ ($a_{m}\neq0$,
$m\leq n$). The completion of $k$ at $\nu_{0}$ is $F=k_{\nu_{y}}=\mathbb{F}_{q}\left(\left(y\right)\right)$,
the field of Laurent polynomials over $\mathbb{F}_{q}$. The ring
of integer of $F$ is $\mathcal{O}=\mathbb{F}_{q}\left[\left[y\right]\right]$,
and we recall that $\mathcal{B}_{d}^{0}\simeq\nicefrac{\mathrm{PGL}_{d}\left(F\right)}{\mathrm{PGL}_{d}\left(\mathcal{O}\right)}$.

We now have $S=\left\{ \nu_{1+y},\nu_{\frac{1}{y}},\nu_{y}\right\} $,
and the ring of $S$-integers in $k$ is $\mathcal{O}_{S}=\mathbb{F}_{q}\left[\frac{1}{1+y},y,\frac{1}{y}\right]$.
As explained above, embedding $\utilde{G}\left(k\right)$ in some
$\mathrm{GL}_{n}\left(k\right)$ gives rise to $\Gamma=\utilde{G}\left(\mathcal{O}_{S}\right)=\utilde{G}\left(k\right)\bigcap\mathrm{GL}_{n}\left(\mathcal{O}_{S}\right)$,
which embeds as a cocompact arithmetic lattice in $\utilde{G}\left(F\right)\simeq\mathrm{PGL}_{d}\left(F\right)$.

Until now we have followed the general construction described in the
beginning of this section. In what follows we describe the Cartwright-Steger
group, a subgroup of $\Gamma$ which acts simply transitively on $\mathcal{B}_{d}^{0}$.

The definition of $\Gamma=\utilde{G}\left(\mathcal{O}_{S}\right)$
involves a choice of an embedding of $\utilde{G}\left(k\right)$ in
$\mathrm{GL}_{n}\left(k\right)$. It turns out that this embedding
can be chosen so that $\Gamma$ is simply $\nicefrac{D\left(\mathcal{O}_{S}\right)^{\times}}{\mathcal{O}_{S}^{\times}}$,
where $D\left(\mathcal{O}_{S}\right)$ stands for the $\mathcal{O}_{S}$-algebra
having the  $\mathcal{O}_{S}$-basis $\left\{ \xi_{i}z^{j}\right\} _{i,j=0}^{d-1}$,
and again the relations $z\xi_{i}=\phi\left(\xi_{i}\right)z$ and
$z^{d}=1+y$ (see \cite[Prop.\ 3.3]{Lubotzky2005b}). Note that as
$z^{d}=1+y$ and $1+y$ is invertible in $\mathcal{O}_{S}$, $z$
is invertible in $D\left(\mathcal{O}_{S}\right)$. Let $b=1-z^{-1}\in D\left(\mathcal{O}_{S}\right)$,
and note that $b$ is also invertible, since it divides $1-z^{-d}=\frac{y}{1+y}$
and $y\in\mathcal{O}_{S}^{\times}$. Also note that $\mathbb{F}_{q^{d}}$
is a subring of $D\left(\mathcal{O}_{S}\right)$ spanned by the $\xi_{i}$'s.
For every $u\in\mathbb{F}_{q^{d}}^{\times}$ denote $b_{u}=ubu^{-1}\in D\left(\mathcal{O}_{S}\right)^{\times}$.
The element $b_{u}$ depends only on the coset of $u$ in $\nicefrac{\mathbb{F}_{q^{d}}^{\times}}{\mathbb{F}_{q}^{\times}}$,
since $\mathbb{F}_{q}\subseteq Z\left(D\left(\mathcal{O}_{S}\right)\right)$.
Denoting by $\overline{b_{u}}$ the image of $b_{u}$ in $\Gamma=\nicefrac{D\left(\mathcal{O}_{S}\right)^{\times}}{\mathcal{O}_{S}^{\times}}$,
this gives us a set of $\frac{q^{d}-1}{q-1}$ elements $\Sigma_{1}=\left\{ \overline{b_{u}}\,\middle|\, u\in\nicefrac{\mathbb{F}_{q^{d}}^{\times}}{\mathbb{F}_{q}^{\times}}\right\} $
in $\Gamma$. Let $\Lambda=\left\langle \Sigma_{1}\right\rangle $.
This is the promised Cartwright-Steger group.
\begin{thm}[{\cite{cartwright1998family}, cf. \cite[Prop.\ 4.8]{Lubotzky2005b}}]
The group $\Lambda$ acts simply transitively on the vertices of
$\mathcal{B}_{d}\left(F\right)$.
\end{thm}
The set $\Sigma_{1}=\left\{ \overline{b_{u}}\,\middle|\, u\in\nicefrac{\mathbb{F}_{q^{d}}^{\times}}{\mathbb{F}_{q}^{\times}}\right\} $
takes the ``initial vertex'' $x_{0}$ of the building (i.e.\ the
equivalence class of the standard lattice) to the set of its neighbors
$x$ with $\tau\left(x\right)=1$ (i.e.\ the neighboring vertices
of color $1$, for which the connecting edge also has color $1$).
These correspond to the codimension one subspaces of $\mathbb{F}_{q}^{d}$
and indeed there are $\frac{q^{d}-1}{q-1}$ such (on which the finite
group $\nicefrac{\mathbb{F}_{q^{d}}^{\times}}{\mathbb{F}_{q}^{\times}}$
acts transitively!) Now, for $i=2,\ldots,d-1$, let $\Sigma_{i}=\left\{ \gamma\in\Lambda\,\middle|\,\tau\left(\left(x_{0},\gamma x_{0}\right)\right)=i\right\} $,
i.e., the subset of $\Lambda$ of those elements which takes $x_{0}$
to a neighbor of color $i$. As $\Lambda$ acts simply transitively
$\left|\Sigma_{i}\right|=\left[\begin{smallmatrix}d\\
i
\end{smallmatrix}\right]_{q}$ where $\left[\begin{smallmatrix}d\\
i
\end{smallmatrix}\right]_{q}$ is the number of subspaces of $\mathbb{F}_{q}^{d}$ of codimension
$i$. Let $\Sigma=\bigcup_{i=1}^{d-1}\Sigma_{i}$. One can deduce
now that the $1$-skeleton of $\mathcal{B}_{d}$ can be identified
with $\mathrm{Cay}\left(\Lambda;\Sigma\right)$. 

Now for every $0\neq I\triangleleft\mathcal{O}_{S}$, we can define
$\Lambda\left(I\right)$ as $\Lambda\left(I\right)=\ker\left(\Lambda\rightarrow\utilde{G}\left(\nicefrac{\mathcal{O}_{S}}{I}\right)\right)$.
This defines a complex $\Lambda\left(I\right)\backslash\mathcal{B}_{d}$
which by Theorem \ref{thm:congruence-ramanujan-complex} is a Ramanujan
complex.

Observe now that the building $\mathcal{B}$ is a clique complex,
namely, a set of $i+1$ vertices forms a simplex if and only if every
two vertices in it form a $1$-edge. In particular, the full structure
of the complex is determined by the $1$-skeleton. The same is true
for the quotients $\Lambda\left(I\right)\backslash\mathcal{B}_{d}$
(at least for large enough quotients, since the map $\mathcal{B}_{d}\rightarrow\Lambda\left(I\right)\backslash\mathcal{B}_{d}$
is a local isomorphism, moreover the injective radius of $\Lambda\left(I\right)\backslash\mathcal{B}_{d}$
grows logarithmically w.r.t.\ its size). So, these complexes are
the Cayley complexes of the group $\nicefrac{\Lambda}{\Lambda\left(I\right)}$
with respect to the set of generators $\Sigma$, or more precisely,
their images in $\nicefrac{\Lambda}{\Lambda\left(I\right)}$. Recall
that a Cayley complex of a group $H$ w.r.t.\ a symmetric set of
generators $\Sigma$ is the simplicial complex for which a subset
$\Delta$ of $H$ is a simplex iff $a^{-1}b\in\Sigma$ for every $a,b\in\Delta$.
This is the clique complex determined by $\mathrm{Cay}\left(H;\Sigma\right)$.

To make all this explicit also in the computer science sense, one
needs to identify the quotients $\nicefrac{\Lambda}{\Lambda\left(I\right)}$.
This is carried out using the Strong Approximation Theorem. When $I$
is a prime ideal of $\mathcal{O}_{S}$, we get that $\nicefrac{\mathcal{O}_{S}}{I}$
is a finite field of order $q^{e}$ for some $e$. The group $\nicefrac{\Lambda}{\Lambda\left(I\right)}$
is then a subgroup of $\mathrm{PGL}_{d}\left(\mathbb{F}_{q^{e}}\right)$
containing $\mathrm{PSL}_{d}\left(\mathbb{F}_{q^{e}}\right)$. Various
choices of ideals $I$ can be made to make sure that any of the subgroups
$H$ between $\mathrm{PSL}_{d}\left(\mathbb{F}_{q^{e}}\right)$ and
$\mathrm{PGL}_{d}\left(\mathbb{F}_{q^{e}}\right)$ can occur. Note
that the quotient $\nicefrac{\mathrm{PGL}_{d}\left(\mathbb{F}_{q^{e}}\right)}{\mathrm{PSL}_{d}\left(\mathbb{F}_{q^{e}}\right)}$
is a cyclic group of order dividing $d$. The resulting graphs are
therefore $t$-partite for some $t\mid d$, just as in case $d=2$
where we have had bi-partite and non-bipartite. We skip the technical
details and give only a corollary (see Theorem 1.1 and Algorithm 9.2
in \cite{Lubotzky2005b}).
\begin{thm}
Let $q$ be a given prime power, $d\geq2$ and $e\geq1$. Assume $q^{e}\geq4d^{2}$.
Every subgroup $H$, with $\mathrm{PSL}_{d}\left(\mathbb{F}_{q^{e}}\right)\leq H\leq\mathrm{PGL}_{d}\left(\mathbb{F}_{q^{e}}\right)$,
has an (explicit) set $\Sigma$ of $\left[\begin{smallmatrix}d\\
1
\end{smallmatrix}\right]_{q}+\ldots+\left[\begin{smallmatrix}d\\
d-1
\end{smallmatrix}\right]_{q}$ generators, such that the Cayley complex of $H$ w.r.t.\ $\Sigma$
is a Ramanujan complex covered by $\mathcal{B}_{d}\left(F\right)$
when $F=\mathbb{F}_{q}\left(\left(y\right)\right)$.
\end{thm}
We mention in passing that the construction in this subsection is
of interest even for $d=2$ (in spite of the fact that we have already
seen other constructions of Ramanujan graphs in the previous chapter)
since for $d=2$, $\nicefrac{\mathbb{F}_{q^{2}}^{\times}}{\mathbb{F}_{q}^{\times}}$
acts transitively on \emph{all }the $q+1$ neighbors of the standard
lattice. From this one can deduce that the resulting Ramanujan graphs
are edge transitive and not merely vertex transitive, as is always
the case for Cayley graphs. This extra symmetry plays a crucial role
in an application to the theory of error correcting codes (see \cite{kaufman2011edge}).

We hope that the higher Ramanujan complexes will also bear some fruits
in combinatorics like their one dimensional counterparts. For first
steps in this direction see \cite{lubotzky2007moore} and \cite{Kaufman2013}.

\section{High dimensional expanders}

In Definition \ref{def:expander} we presented the definition of expanding
graphs. In recent years several suggestions have been proposed as
to what should be the ``right'' definition of ``expander'' for
higher dimensional simplicial complexes. In this chapter we will bring
some of these as well as few results about the relations between them.
This area is still in its primal state, and we can expect more developments.
The importance of expanding graphs suggests that studying expanding
simplicial complexes will also turn out to be very fruitful.

\subsection{\label{sub:Simplicial-complexes-and-cohom}Simplicial complexes and
cohomology}

A finite simplicial complex $X$ is a finite collection of subsets
of a set $X^{\left(0\right)}$, called the set of vertices of $X$,
which is closed under taking subsets. The sets in $X$ are called
\emph{simplices} or \emph{faces} and we denote by $X^{\left(i\right)}$
the set of simplices of $X$ of dimension $i$, which are the sets
of $X$ of size $i+1$. So $X^{\left(-1\right)}$ is comprised of
the empty set, $X^{\left(0\right)}$ - of the vertices, $X^{\left(1\right)}$
- the edges, $X^{\left(2\right)}$ - the triangles, etc. Throughout
this discussion we will assume that $X^{\left(0\right)}=\left\{ v_{1},\ldots,v_{n}\right\} $
is the set of vertices and we fix the order $v_{1}<v_{2}<\ldots<v_{n}$
among the vertices. Now, if $F\in X^{\left(i\right)}$ we write $F=\left\{ v_{j_{0}},\ldots,v_{j_{i}}\right\} $
with $v_{j_{0}}<v_{j_{1}}<\ldots<v_{j_{i}}$. If $G\in X^{\left(i-1\right)}$,
we denote the \emph{oriented incidence number} $\left[F:G\right]$
by $\left(-1\right)^{\ell}$ if $F\backslash G=\left\{ v_{j_{\ell}}\right\} $
and $0$ if $G\nsubseteq F$. In particular for every vertex $v\in X^{\left(0\right)}$
and for the unique face $\varnothing\in X^{\left(-1\right)}$, $\left[v:\varnothing\right]=1$.

If $\mathbb{F}$ is a field then $C^{i}\left(X,\mathbb{F}\right)$
is the $\mathbb{F}$-vector space of the functions from $X^{\left(i\right)}$
to $\mathbb{F}$. This is a vectors space of dimension $\left|X^{\left(i\right)}\right|$
over $\mathbb{F}$ where the characteristic functions $\left\{ e_{F}\,\middle|\, F\in X^{\left(i\right)}\right\} $
serve as a basis.

The coboundary map $\delta_{i}:C^{i}\left(X,\mathbb{F}\right)\rightarrow C^{i+1}\left(X,\mathbb{F}\right)$
is given by:
\[
\left(\delta_{i}f\right)\left(F\right)=\sum_{G\in X^{\left(i\right)}}\left[F:G\right]f\left(G\right)
\]
so if $f=e_{G}$ for some $G\in X^{\left(i\right)}$, $\delta_{i}e_{G}$
is a sum of all the simplices of dimension $i+1$ containing $G$
with signs $\pm1$ according to the relative orientations.

It is well known and easy to prove that $\delta_{i}\circ\delta_{i-1}=0$.
Thus $B^{i}\left(X,\mathbb{F}\right)=\im\delta_{i-1}$ - ``the space
of $i$-coboundaries'' is contained in $Z^{i}\left(X,\mathbb{F}\right)=\ker\delta_{i}$
- the $i$-cocycles and the quotient $H^{i}\left(X,\mathbb{F}\right)=\nicefrac{Z^{i}\left(X,\mathbb{F}\right)}{B^{i}\left(X,\mathbb{F}\right)}$
is the $i$-th cohomology group of $X$ over $\mathbb{F}$.

In a dual way one can look at $C_{i}\left(X,\mathbb{F}\right)$ -
the $\mathbb{F}$-vector space spanned by the simplices of dimension
$i$. Let $\partial_{i}:C_{i}\left(X,\mathbb{F}\right)\rightarrow C_{i-1}\left(X,\mathbb{F}\right)$
be the boundary map defined on the basis element $F$ by: $\partial F=\sum_{G\in X^{\left(i-1\right)}}\left[F:G\right]\cdot G$,
i.e.\ if $F=\left\{ v_{j_{0}},\ldots,v_{j_{i}}\right\} $ then $\partial_{i}F=\sum_{t=0}^{i}\left(-1\right)^{t}\left\{ v_{j_{0}},\ldots,\widehat{v_{j_{t}}},\ldots,v_{j_{i}}\right\} $.
Again $\partial_{i}\circ\partial_{i+1}=0$ and so the boundaries $B_{i}\left(X,\mathbb{F}\right)=\im\partial_{i+1}$
are inside the cycles $Z_{i}\left(X,\mathbb{F}\right)=\ker\partial_{i}$
and $H_{i}\left(X,\mathbb{F}\right)=\nicefrac{Z_{i}\left(X,\mathbb{F}\right)}{B_{i}\left(X,\mathbb{F}\right)}$
gives the $i$-th homology group of $X$ over $\mathbb{F}$. As $\mathbb{F}$
is a field, it is not difficult in this case to show that $H_{i}\left(X,\mathbb{F}\right)\simeq H^{i}\left(X,\mathbb{F}\right)$.
Sometimes, it is convenient to identify $C_{i}\left(X,\mathbb{F}\right)$
and $C^{i}\left(X,\mathbb{F}\right)$ by assigning $F$ to $e_{F}$.

The $i$-th laplacian of $X$ over $F$ is defined as the linear operator
$\Delta_{i}:C^{i}\left(X,\mathbb{F}\right)\rightarrow C^{i}\left(X,\mathbb{F}\right)$
given by $\Delta_{i}=\partial_{i+1}\delta_{i}+\delta_{i-1}\partial_{i}$.
The operator $\partial_{i+1}\delta_{i}$ is sometimes denoted (for
clear reasons!) $\Delta_{i}^{up}$, while $\Delta_{i}^{down}=\delta_{i-1}\partial_{i}$.
In fact, $\partial_{i+1}$ is the dual of $\delta_{i}$ and so the
eigenvalues of $\Delta_{i}^{up}$ and $\Delta_{i+1}^{down}$ differ
only by the multiplicity of zero. Note that what is customarily called
the laplacian of a graph is actually the upper $0$-laplacian: 
\[
\Delta_{0}^{up}f\left(x\right)=\deg\left(x\right)f\left(x\right)-\sum_{y\sim x}f\left(y\right).
\]

\subsection{$\mathbb{F}_{2}$-coboundary expansion}

It seems that the first definition of higher dimensional expansion
was given by Linial-Meshulam \cite{Linial2006}, Meshulam-Wallach
\cite{meshulam2009homological} and Gromov \cite{Gromov2010} (see
also \cite{dotterrer2010coboundary,Gundert2012,mukherjee2012cheeger,newman2012multiplicative})
as follows
\begin{defn}
For a simplicial complex $X$, the \emph{$\mathbb{F}_{2}$-coboundary
expansion of $X$ in dimension $i$} is 
\[
\mathcal{E}_{i}\left(X\right)=\min\left\{ \dfrac{\left\Vert \delta_{i-1}f\right\Vert }{\left\Vert \left[f\right]\right\Vert }\,\middle|\, f\in C^{i-1}\left(X,\mathbb{F}_{2}\right)\backslash B^{i-1}\left(X,\mathbb{F}_{2}\right)\right\} .
\]

\end{defn}
In other papers this notion is referred to as ``cohomological expansion'',
``coboundary expansion'', or ``combinatorial expansion''. Let
us explain the notation here: $\mathbb{F}_{2}$ is the field of order
two, for $f\in C^{i-1}$ (and similarly for $\delta f\in C^{i}$),
$\left\Vert f\right\Vert $ is simply the number of $\left(i-1\right)$-simplices
$F$ for which $f\left(F\right)\neq0$. Finally, $\left[f\right]$
is the coset $f+B^{i-1}\left(X,\mathbb{F}_{2}\right)$ and 
\[
\left\Vert \left[f\right]\right\Vert =\min\left\{ \left\Vert g\right\Vert \,\middle|\, g\in\left[f\right]\right\} =\min\left\{ \left\Vert f+\delta_{i-2}h\right\Vert \,\middle|\, h\in C^{i-2}\left(X,\mathbb{F}_{2}\right)\right\} .
\]
One can see that $\left\Vert \left[f\right]\right\Vert $ is the minimal
distance of $f$ from $B^{i-1}\left(X,\mathbb{F}_{2}\right)$ in the
Hamming metric, and in particular that $\left\Vert \left[f\right]\right\Vert =0$
iff $f\in B^{i-1}\left(X,\mathbb{F}_{2}\right)$. 

Some authors prefer to normalize the expansion as follows:
\[
\widetilde{\mathcal{E}}_{i}\left(X\right)=\min\left\{ \dfrac{\nicefrac{\left\Vert \delta_{i-1}f\right\Vert }{\left|X^{\left(i\right)}\right|}}{\nicefrac{\left\Vert \left[f\right]\right\Vert }{\left|X^{\left(i-1\right)}\right|}}\,\middle|\, f\in C^{i-1}\left(X,\mathbb{F}_{2}\right)\backslash B^{i-1}\left(X,\mathbb{F}_{2}\right)\right\} .
\]

Let us explain why this artificially looking definition exactly gives
expander graphs in the one dimensional case: If $X$ is a graph then
$B^{0}=\im\delta_{-1}$ is the one dimensional space containing two
functions, the zero function $0$ and the constant function $\mathbbm{1}$
on all the vertices of $X$. Now, if $f\in C^{0}\left(X,\mathbb{F}_{2}\right)$
then $f$ is nothing more than the characteristic function $\chi_{A}$
of some subset $A\subseteq X^{\left(0\right)}$, in which case $\left[f\right]=f+B^{0}\left(X,\mathbb{F}_{2}\right)=\left\{ \chi_{A},\chi_{\overline{A}}\right\} $
where $\overline{A}$ is the complement of $A$ in $X^{\left(0\right)}$.
Thus $\left\Vert \left[f\right]\right\Vert =\min\left(\left|A\right|,\left|\overline{A}\right|\right)$.
Finally $\left\Vert \delta f\right\Vert $ is nothing more than the
size of $E\left(A,\overline{A}\right)$, i.e., the set of edges between
$A$ and $\overline{A}$. We can now see that the $\mathbb{F}_{2}$-coboundary
expansion $\mathcal{E}_{1}\left(X\right)$ (which is the only relevant
dimension in this case) is exactly $\overline{h}\left(X\right)$ as
in Remark \ref{rem:old-cheeger}.

Very few results have been proven so far about this concept. Here
is one of them (see \cite{meshulam2009homological,Gromov2010})
\begin{prop}
The complete complex $\Delta_{\left[n-1\right]}$, the simplicial
complex on $n$ vertices where every subset is a face, has $\mathbb{F}_{2}$-coboundary
expansion $\mathcal{E}_{i}\left(\Delta_{\left[n-1\right]}\right)\geq\frac{n}{i+1}$,
$1\leq i\leq n-1$. Equivalently, $\widetilde{\mathcal{E}}_{i}\left(\Delta_{\left[n-1\right]}\right)\geq1$
(in fact, it converges to $1$ when $n$ grows to $\infty$).\end{prop}
\begin{rem}
\label{rem:cohom-expansion}One should note that $X$ has positive
$\mathbb{F}_{2}$-coboundary expansion in dimension $i$ if and only
if $H^{i-1}\left(X,\mathbb{F}_{2}\right)=0$: If $Z^{i-1}\left(X,\mathbb{F}_{2}\right)=B^{i-1}\left(X,\mathbb{F}_{2}\right)$
then $\delta f\neq0$ for every $f\in C^{i-1}\backslash B^{i-1}$,
while if $f\in Z^{i-1}\left(X,\mathbb{F}_{2}\right)\backslash B^{i-1}\left(X,\mathbb{F}_{2}\right)$
then $\delta f=0$ and $\left\Vert \left[f\right]\right\Vert \neq0$.
This vanishing of $H^{d-1}\left(X,\mathbb{F}_{2}\right)$ in the graph
case, $d=1$, is the vanishing of $H^{0}\left(X,\mathbb{F}_{2}\right)$
which exactly means that the graph $X$ is connected. Indeed, it is
clear that an $\varepsilon$-expander graph is connected.
\end{rem}
Most of the known results on coboundary expansion refer to complexes
$X$ of dimension $d$ whose $d-1$ skeleton is complete (i.e.\ every
subset of $X^{\left(0\right)}$ of size $d$ is a face in the complex).
See \cite{Linial2006,meshulam2009homological,Gromov2010,dotterrer2010coboundary,wagner2011minors,Gundert2012}
for various results, mainly on random complexes. 

As far as we know there is no known family of higher dimensional $\mathbb{Z}_{2}$-coboundary
expanders of \emph{bounded} degree (i.e.\ where the number of faces
is linear in the number of vertices). It is natural to suggest that
the Ramanujan complexes of §2 (and even more generally, all finite
quotients of higher dimensional Bruhat-Tits buildings of simple groups
of rank $\geq2$ over local fields) are such. But this is not the
case in general. For example, let $\Gamma$ be any cocompact lattice
in $\mathrm{PGL}_{3}\left(F\right)$ where $F$ is a local field and
assume $\nicefrac{\Gamma}{\left[\Gamma,\Gamma\right]\Gamma^{2}}$
is non-trivial (i.e.\ $\Gamma$ has a non-trivial abelian quotient
of $2$-power order - by \cite{lubotzky1987finite} every lattice
has such a sublattice of finite index) then $H^{1}\left(\Gamma\backslash\mathcal{B},\mathbb{F}_{2}\right)\neq0$
(since $\mathcal{B}$ - the Bruhat-Tits building of $\mathrm{PGL}_{3}\left(F\right)$
is contractible) and so by Remark \ref{rem:cohom-expansion} the $\mathbb{F}_{2}$-coboundary
expansion of $X=\Gamma\backslash\mathcal{B}$ in dimension $2$ is
$0$. It might be that the vanishing of the cohomology is the only
obstruction. 

Another possible way to circumvent this is to use instead the notion
of Gromov of ``filling'': The filling of $X$ (in dimension $i$)
is 
\[
\nu_{i}\left(X\right)=\max\left\{ \frac{\left\Vert f+Z^{i-1}\right\Vert }{\left\Vert \delta_{i-1}f\right\Vert }\,\middle|\, f\in C^{i-1}\left(X,\mathbb{F}_{2}\right)\backslash Z^{i-1}\left(X,\mathbb{F}_{2}\right)\right\} .
\]
When $H^{i-1}\left(X,\mathbb{F}_{2}\right)$ vanishes, the filling
and the $\mathbb{F}_{2}$-coboundary expansion are related by $\nu_{i}\left(X\right)=\frac{1}{\mathcal{E}_{i}\left(X\right)}=\left[\min\limits _{f\in C^{i-1}\backslash B^{i-1}}\frac{\left\Vert \delta_{i-1}f\right\Vert }{\left\Vert f+B^{i-1}\right\Vert }\right]^{-1}$.
When $H^{i-1}\left(X,\mathbb{F}_{2}\right)$ does not vanish, $\mathcal{E}_{i}\left(X\right)$
is zero (see \ref{rem:cohom-expansion}), but $\nu_{i}\left(x\right)$
is always finite since $\left\Vert \delta_{i-1}f\right\Vert \neq0$
for $f\notin Z^{i-1}$. For example, the Cheeger constant $\overline{h}$
vanishes for a disconnected graph, while $\frac{1}{\nu_{1}\left(x\right)}$
is the \emph{mediant} (or ``freshman sum'') of the Cheeger constants
of the connected components of the graph, and it is always positive.
We present the following conjecture:
\begin{conjecture}
Let $\mathcal{B}$ be the Bruhat-Tits building associated with $\mathrm{PGL}_{d}\left(F\right)$,
$F$ a local field and $d\geq3$. There exists a constant $\nu=\nu\left(d,F\right)$
such that $\nu_{i}\left(X\right)\leq\nu$ for every finite quotient
$X$ of $\mathcal{B}$.
\end{conjecture}
Even special cases of this conjecture (e.g.\ the case $d=3$ and
$q$, the residue field of $F$, large) are of importance in coding
theory as shown in \cite{Kaufman2013}.

\subsection{The Cheeger constant}

The Cheeger constant $h\left(X\right)$ for a graph $X$ is defined
in Definition \ref{def:Cheeger} above (see also Remark \ref{rem:old-cheeger}
there). One may argue what should be the right definition of $h\left(X\right)$
when $X$ is a higher dimensional simplicial complex. Let us follow
here the definition given in \cite{parzanchevski2012isoperimetric}:
\begin{defn}
\label{def:high-cheeger}For a $d$-dimensional simplicial complex
$X$, denote
\[
h\left(X\right)=\min_{X^{\left(0\right)}=\coprod\limits _{i=0}^{d}A_{i}}\frac{\left|X^{\left(0\right)}\right|\left|F\left(A_{0},\ldots,A_{d}\right)\right|}{\left|A_{0}\right|\cdot\ldots\cdot\left|A_{d}\right|}
\]
where the minimum is over all the partitions of $X^{\left(0\right)}$
into nonempty sets $A_{0},\ldots,A_{d}$ and $F\left(A_{0},\ldots,A_{d}\right)$
denotes the set of $d$-dimensional simplices with exactly one vertex
in each $A_{i}$.
\end{defn}
For $d=1$, it coincides with Definition \ref{def:Cheeger}. But,
in a way, this definition keeps the spirit of the mixing lemma (Proposition
\ref{prop:mixing-lemma}): $h\left(X\right)$ measures the number
of ``edges'' (i.e.\ $d$-faces) ``between'' (i.e.\ with a single
representative in each of) the $A_{i}$. The quantity $\left|F\left(A_{0},\ldots,A_{d}\right)\right|$
is ``normalized'' by multiplying it by $\frac{\left|X^{\left(0\right)}\right|}{\prod_{i=0}^{d}\left|A_{i}\right|}$.
This definition works well when $X$ has a complete $\left(d-1\right)$-skeleton
(see more in §\ref{sub:The-overlap-property}), but as noted in \cite{parzanchevski2012isoperimetric}
it gives zero whenever $X^{\left(d-1\right)}$ is not complete (If
$G=\left\{ v_{0},\ldots,v_{d-1}\right\} \notin X^{\left(d-1\right)}$
take $A_{i}=\left\{ v_{i}\right\} $ for $i=0,\ldots,d-1$ and $A_{d}=X^{\left(0\right)}\backslash G$.
Then $F\left(A_{0},\ldots,A_{d}\right)=\varnothing$).\textbf{ }In
\cite{parzanchevski2012isoperimetric}, the authors call the difference
\begin{equation}
\left|\left|F\left(A_{0},\ldots,A_{d}\right)\right|-\frac{\left|X^{\left(d\right)}\right|\left|A_{0}\right|\cdot\ldots\cdot\left|A_{d}\right|}{{n \choose d+1}}\right|\label{eq:discrepancy}
\end{equation}
the \emph{discrepancy} of $A_{0},\ldots,A_{d}$, and they bound this
value, and the constant $h\left(X\right)$, in terms of the spectrum
of the laplacian. This brings us to our next subject.

\subsection{\label{sub:Spectral-gap}Spectral gap}

In §1 we saw that the notion of expander can be described by means
of the eigenvalues of the adjacency matrix $A$ of the graph. For
a $k$-regular graph $X$, the matrix $A$ is nothing more than $A=kI-\Delta_{0}^{up}$
where $\Delta_{0}^{up}$ is the $0$-dimensional upper laplacian of
$X$ over $\mathbb{F}=\mathbb{R}$ as defined in §\ref{sub:Simplicial-complexes-and-cohom}.
We can translate Theorem \ref{thm:discrete-cheeger} to deduce that
a family of $k$-regular graphs $\left\{ X_{t}\right\} _{t\in I}$
is a family of expanders iff there exists $\varepsilon>0$ such that
every eigenvalue $\lambda$ of $\Delta_{0}^{up}\Big|_{Z_{0}\left(X,\mathbb{R}\right)}=\Delta_{0}\Big|_{Z_{0}\left(X,\mathbb{R}\right)}$
satisfies $\lambda\geq\varepsilon$ (the last is equality is since
$\Delta_{i}^{down}\left(Z_{i}\left(X,\mathbb{R}\right)\right)=\delta_{i-1}\partial_{i}\left(\ker\partial_{i}\right)=0$).
Note that $Z_{0}\left(X,\mathbb{R}\right)=\left\{ f:X^{\left(0\right)}\rightarrow\mathbb{R}\,\middle|\,\sum_{x\in X^{\left(0\right)}}f\left(x\right)=0\right\} $.
It is therefore natural to generalize and to define
\begin{defn}
Let $X$ be a simplicial complex of dimension $d$ and $0\leq i\leq d-1$.
We denote $\lambda_{i}\left(X\right)=\min\Spec\left(\Delta_{i}\Big|_{Z_{i}\left(X,\mathbb{R}\right)}\right)$
and we say that $X$ has spectral gap $\lambda_{i}\left(X\right)$
in dimension $i$. We write $\lambda\left(X\right)$ for $\lambda_{d-1}\left(X\right)$. 
\end{defn}
It is natural to expect that just like in graphs where there is a
direct connection between the Cheeger constant and the spectral gap,
something like that should happen in the higher dimensional case,
but examples presented in \cite{parzanchevski2012isoperimetric} show
that there exist simplicial complexes with $\lambda\left(X\right)=0$
while $h\left(X\right)>0$. The mystery has been revealed recently
in \cite{parzanchevski2012isoperimetric} where it is shown that the
right generalization of the Cheeger inequalities is:
\begin{thm}[\cite{parzanchevski2012isoperimetric}]
\label{thm:high-cheeger}Let $X$ be a finite $d$-dimensional complex
with a complete $\left(d-1\right)$-skeleton. If $k$ is the maximal
degree of a $\left(d-1\right)$-cell, then
\[
\frac{d\left(1-\frac{d-1}{\left|X^{\left(0\right)}\right|}\right)^{2}}{8k}h^{2}\left(X\right)-\left(d-1\right)k\leq\lambda\left(X\right)\leq h\left(X\right).
\]

\end{thm}
The reader may note that this Theorem, when specialized to $d=1$,
gives \emph{exactly} Theorem \ref{thm:discrete-cheeger}. 

A similar generalization is obtained in \cite{parzanchevski2012isoperimetric}
for the expander mixing lemma (Proposition \ref{prop:mixing-lemma}
above). Given any two sets of vertices $A,B\subseteq V$, the mixing
lemma for graphs bounds the deviation of $\left|E\left(A,B\right)\right|$
from its expected value in a random $k$-regular graph, in terms of
the spectral invariant $\mu_{0}$. From the perspective of the simplicial
laplacian, $\mu_{0}$ is the spectral radius of $kI-\Delta_{0}\Big|_{Z_{0}\left(X,\mathbb{R}\right)}$,
i.e.\ the maximal absolute value of its eigenvalues. The following
generalization then holds for higher dimensional complexes:
\begin{thm}[\cite{parzanchevski2012isoperimetric}]
\label{thm:high-mixing}Let $X$ be a finite $d$-dimensional complex
with a complete $\left(d-1\right)$-skeleton. Let $k$ be the average
degree of a $\left(d-1\right)$-cell, and define
\[
\mu_{0}\left(X\right)=\max\left\{ \left|\gamma\right|\,\middle|\,\gamma\in\Spec\left(kI-\Delta_{d-1}\Big|_{Z_{d-1}\left(X,\mathbb{R}\right)}\right)\right\} .
\]
Then for every disjoint sets of vertices $A_{0},\ldots,A_{d}$,
\[
\left|\left|F\left(A_{0},\ldots,A_{d}\right)\right|-\frac{k\left|A_{0}\right|\cdot\ldots\cdot\left|A_{d}\right|}{\left|X^{\left(0\right)}\right|}\right|\leq\mu_{0}\left(X\right)\left(\left|A_{0}\right|\cdot\ldots\cdot\left|A_{d}\right|\right)^{\frac{d}{d+1}}.
\]

\end{thm}
Again, when specializing to $d=1$ this gives the original expander
mixing lemma for graphs, except for the additional assumption that
the sets of vertices are disjoint. The reader is referred to \cite{parzanchevski2012isoperimetric}
for the proofs of Theorems \ref{thm:high-cheeger} and \ref{thm:high-mixing}.
So far, the results are under the assumption of full $\left(d-1\right)$-skeleton
but a work on the general situation is in progress.

It is natural to suggest some extension of Alon-Boppana theorem (Theorem
\ref{thm:Alon-Boppana}) to this high dimensional case (see also Theorem
\ref{thm:Li-AlonBoppana}). In \cite{Parzanchevski} it is shown that
the high dimensional analogue of Alon-Boppana indeed holds in several
interesting cases (for example, for quotients of an infinite complex
with nonzero spectral gap), but that it can also fail. The reader
is referred to that interesting work for more details.

The most important work so far on the spectral gap of complexes is
the seminal work of Garland \cite{Garland1973}. As this work has
been described in many placed (e.g.\ \cite{borel1973cohomologie,Zuk1996,Gundert2012})
we will not elaborate on it here. We just mention that Garland proved
Serre's conjecture that $H^{i}\left(X,\mathbb{R}\right)=0$ for every
$1\leq i\leq d-1$ where $X$ is a finite quotient of the Bruhat-Tits
building of a simple group of rank $d\geq2$ over a local field $\mathbb{F}$.
He did this by proving a bound on the spectral gaps which depends
only $d$ and $\mathbb{F}$ (the $i$-th cohomology group over $\mathbb{R}$
vanishes iff the corresponding spectral gap $\lambda_{i}$ is nonzero).

It is still not clear what is the relation between the coboundary
expansion and the spectral gap. See \cite{Gundert2012,mukherjee2012cheeger}
where some complexes are presented with $\lambda{}_{i}\left(X\right)$
arbitrarily small while $\mathcal{E}_{i}\left(X\right)$ is bounded
away from zero, and also the other way around.

\subsection{\label{sub:The-overlap-property}The overlap property}

An interesting ``overlap'' property for complexes, which is closely
related to expanders, was defined by Gromov \cite{Gromov2010}, and
was further studied in \cite{Fox2011,Matouvsek2011,karasev2012simpler}.
We need first some notation: Let $X$ be a $d$-dimensional simplicial
complex and $\varphi:X^{\left(0\right)}\rightarrow\mathbb{R}^{d}$
an injective map. The map $\varphi$ can be extended uniquely to a
simplicial mapping $\widetilde{\varphi}$ from $X$ (considered now
as a topological space in the obvious way) to $\mathbb{R}^{d}$ (i.e.\ by
extending $\varphi$ affinely to the edges, triangles, etc.) This
will be called a geometric extension. The map $\varphi$ can be extended
in many different ways to a continuous map $\widetilde{\varphi}$
from the topological simplicial complex $X$ to $\mathbb{R}^{d}$,
such $\widetilde{\varphi}$ will be called topological extensions.
\begin{defn}
Let $X$ be a $d$-dimensional simplicial complex and $0<\varepsilon\in\mathbb{R}$.
We say that $X$ has \emph{$\varepsilon$-geometric overlap} (resp.\ \emph{$\varepsilon$-topological
overlap}) if for every injective map $\varphi:X^{\left(0\right)}\rightarrow\mathbb{R}^{d}$
and a geometric (resp.\ topological) extension $\widetilde{\varphi}:X\rightarrow\mathbb{R}^{d}$,
there exists a point $z\in\mathbb{R}^{d}$ such that $\widetilde{\varphi}^{-1}\left(z\right)$
intersects at least $\varepsilon\cdot\left|X^{\left(d\right)}\right|$
of the $d$-dimensional simplices of $X$.
\end{defn}
To digest this definition, let us spell out what does this means for
expander graphs: Let $\varphi:X^{\left(0\right)}\rightarrow\mathbb{R}$
be an injective map and $\widetilde{\varphi}$ any continuous extension
of it to the graph. Let $z\in\mathbb{R}$ be a point such that $\left\lfloor \frac{1}{2}\left|X^{\left(0\right)}\right|\right\rfloor $
of the images of the vertices are above it (and call $L\subseteq X^{\left(0\right)}$
this set of vertices) and the rest are below it. Then $\widetilde{\varphi}^{-1}\left(z\right)$
intersects all the edges of $E\left(L,\overline{L}\right)$ (= the
set of edges going from $L$ to its complement). If $X$ is an $\varepsilon$-expander
$k$-regular graph, then $X^{\left(1\right)}=\frac{\left|X^{\left(0\right)}\right|k}{2}$
while $\left|E\left(L,\overline{L}\right)\right|\geq\frac{\varepsilon}{2}\left|L\right|\approx\frac{\varepsilon}{2}\frac{\left|X^{\left(0\right)}\right|}{2}=\frac{\varepsilon}{2k}\left|X^{\left(1\right)}\right|$.
Thus $X$ has the $\frac{\varepsilon}{2k}$-topological overlapping
property.

The reader should notice however that this property is \emph{not }equivalent
to expander. In fact, it does not even imply that the graph $X$ is
connected. It can be a union of a large expanding graph and a small
connected component. Still, this property captures the nature of expansion
especially in the higher dimensional case.

It is interesting to mention that while it is trivial to prove that
the complete graph is an expander, it is a non-trivial result that
the higher dimensional complete complexes have the overlap property.
This was proved for the geometric overlap in \cite{boros1984number}
for dim $2$ and in \cite{barany1982generalization} for all dimensions.
For the topological overlap, this was proved in \cite{Gromov2010}
(see also \cite{Matouvsek2011,karasev2012simpler}).

The main result of \cite{Fox2011} asserts that there even exist simplicial
complexes of bounded degree with the geometric overlapping property.
They prove it by two methods: probabilistic and constructive. The
constructive examples are the Ramanujan complexes which were discussed
in length in §2 (but under the assumption that $q$ is large enough
w.r.t.\ $d$). In fact, the proof there is valid for all the finite
quotients of $\mathcal{B}=\mathcal{B}\left(\mathrm{PGL}_{d}\left(\mathbb{F}\right)\right)$
and not only to the Ramanujan ones (again assuming $q>>d$). It is
quite likely that the same result holds also for the other Bruhat-Tits
buildings of simple groups of rank $\geq2$.

In all these results the following theorem of Pach plays a crucial
role:
\begin{thm}[\cite{Pach1998}]
For every $d\geq1$, there exists $c_{d}>0$ such that for every
$d+1$ disjoint subsets $P_{1},\ldots,P_{d+1}$ of $n$ points in
general position in $\mathbb{R}^{d}$, there exists $z\in\mathbb{R}^{d}$
and subsets $Q_{i}\subseteq P_{i}$ with $\left|Q_{i}\right|\geq c_{d}\left|P_{i}\right|$
such that every $d$-dimensional simplex with exactly one vertex in
each $Q_{i}$, contains $z$.
\end{thm}
Let us show now, following \cite{parzanchevski2012isoperimetric}
how to deduce the geometric overlap property from Pach's theorem and
the mixing lemma, when we have a ``concentration of the spectrum''.
Let $X$ be a $d$-dimensional complex on $n$ vertices, with a complete
$\left(d-1\right)$ skeleton. For an arbitrary injective map $\varphi:X^{\left(0\right)}\rightarrow\mathbb{R}^{d}$
we can divide $\varphi\left(X^{\left(0\right)}\right)$ to $\left(d+1\right)$-disjoint
sets $P_{0},\ldots,P_{d}$, each of order (approximately) $\frac{n}{d+1}$.
By Pach's theorem there is a point $z\in\mathbb{R}^{d}$ and subsets
$Q_{i}\subseteq P_{i}$ of sizes $\left|Q_{i}\right|=\frac{c_{d}n}{d+1}$,
such that $z$ belongs to every $d$-simplex formed by representatives
from $Q_{0},\ldots,Q_{d}$. This means that for the geometric extension
$\widetilde{\varphi}:X\rightarrow\mathbb{R}^{d}$, $\widetilde{\varphi}^{-1}\left(z\right)$
intersects every simplex in $F\left(\varphi^{-1}\left(Q_{0}\right),\ldots,\varphi^{-1}\left(Q_{d}\right)\right)$.
Turning to the mixing lemma (Theorem \ref{thm:high-mixing} above),
if the average degree of a $\left(d-1\right)$-cell in $X$ is $k$,
and $\Spec\Delta_{d-1}\Big|_{Z_{d-1}\left(X,\mathbb{R}\right)}\subseteq\left[k-\varepsilon,k+\varepsilon\right]$,
then
\begin{align*}
\left|F\left(\varphi^{-1}\left(Q_{0}\right),\ldots,\varphi^{-1}\left(Q_{d}\right)\right)\right| & \geq\frac{k\left|Q_{0}\right|\ldots\left|Q_{d}\right|}{n}-\varepsilon\left(\left|Q_{0}\right|\ldots\left|Q_{d}\right|\right)^{\frac{d}{d+1}}\\
 & =\left(\frac{c_{d}n}{d+1}\right)^{d}\left(\frac{kc_{d}}{d+1}-\varepsilon\right)\cdot
\end{align*}
Since this applies to every $\varphi:X^{\left(0\right)}\rightarrow\mathbb{R}^{d}$,
the quotient by $\left|X^{d}\right|=\frac{k}{d+1}{n \choose d}$ gives
a lower bound for the geometric expansion of $X$: 
\[
\mathrm{overlap}\left(X\right)\geq\frac{\left(\frac{c_{d}n}{d+1}\right)^{d}\left(\frac{kc_{d}}{d+1}-\varepsilon\right)}{\left|X^{d}\right|}\geq\frac{c_{d}^{d}}{e^{d+1}}\left(c_{d}-\frac{\varepsilon\left(d+1\right)}{k}\right).
\]

This is used in \cite{parzanchevski2012isoperimetric} to establish
the overlap property for random complexes in the Linial-Meshulam model
\cite{Linial2006}: It is shown that if the expected degree of a $\left(d-1\right)$-cell
grows logarithmically in the number of vertices then the complexes
have geometric overlap asymptotically almost surely. 

While bounds on the spectrum give some geometric overlap properties,
it is much more difficult to get the topological overlap property.
The only result known to us is the following Theorem of Gromov (see
\cite{Matouvsek2011} for a simplified proof; though still highly
non-trivial):
\begin{thm}
If $X$ has normalized $\mathbb{F}_{2}$-coboundary expansion $\mathcal{\widetilde{E}}_{i}\left(X\right)\geq\varepsilon_{i}$
for all $1\leq i\leq d$ then $X$ has the $\varepsilon$-topological
overlap property for some $\varepsilon=\varepsilon\left(\varepsilon_{1},\ldots,\varepsilon_{d},d\right)>0$.
\end{thm}
Still, we do not know any example of higher dimensional complexes
of bounded degree with the $\mathbb{F}_{2}$-coboundary expansion
property. It is tempting to conjecture that the finite quotients $X$
of a fixed high-rank Bruhat-Tits building of dimension $d$, with
trivial cohomology over $\mathbb{F}_{2}$, form such a family. We
end with this question which seems fundamental for further progress.

\bibliographystyle{alpha}
\bibliography{Japan}
 
\end{document}